\documentclass[12pt,a4paper]{amsart}
\usepackage{mathrsfs}
\usepackage{amsfonts}
\usepackage{txfonts}
\usepackage{color}
\usepackage{hyperref}
\usepackage{latexsym}
\usepackage{amssymb}

\usepackage{refcheck}
\usepackage{comment}

\newtheorem{theorem}{Theorem}[section]
\newtheorem{lemma}[theorem]{Lemma}
\newtheorem{corollary}[theorem]{Corollary}
\newtheorem{proposition}[theorem]{Proposition}
\newtheorem{example}[theorem]{Example}

\theoremstyle{definition}
\newtheorem{definition}[theorem]{Definition}

\newtheorem{remark}[theorem]{Remark}

\newcommand{\R}{\mathbb{R}}
\newcommand{\N}{\mathbb{N}}
\newcommand{\eps}{\varepsilon}

\numberwithin{equation}{section}

\begin{document}

\title[On convex bodies]
{{\bf  On strengthened versions of Klee's convex body problem  in Banach spaces}}

\author{ Lixin Cheng, Wuyi He, Chulei Liu, Zhizheng Yu }
\address{  Lixin Cheng$^\dag$:  School of Mathematical Sciences, Xiamen University,
 Xiamen 361005, China}
\address{ Wuyi He$^\ddag$:  School of Mathematics and Statistics, Chongqing University of Posts and Telecommunications,Chongqing 400065, China}
\address{Chulei Liu$^\natural$: School of Mathematical Sciences, Xiamen University,
 Xiamen 361005, China}
 \address{Zhizheng Yu$^\sharp$: Institute for  Advanced Study in  Mathematics of HIT, Harbin Institute of Technology, Harbin 150001, China}
 \email{$^\dag$: lxcheng@xmu.edu.cn\;\;(Lixin Cheng)}
 \email{$^\ddag$: hewy@cqupt.edu.cn\;\;(Wuyi He)}
 \email{$^\natural$: chuleiliu@stu.xmu.edu.cn \;\;(Chulei Liu) }
 \email{$^\sharp$: zhizheng.yu@hit.edu.cn \;\;(Zhizheng Yu) }

\thanks{$^\dag$ Support in partial
by the Natural Science Foundation of China, grant no. 12271453.         }

\date{}

\begin{abstract}
 In a recent article, Cheng, Jiang and Yuan gave an affirmative answer to Klee's convex bodies problem of Banach spaces in the sense of strict convexity and G\^{a}teaux smoothness. In this paper, we continue to study this problem in strong senses, such as local uniform convexity, uniform convexity, Fr\'{e}chet smoothness and uniform smoothness.  As a result, we show
  (1) Every convex body in a Banach space $X$ is  approximated by locally uniformly convex bodies with respect to the Hausdorff metric if and only if $X$  admits an equivalent locally uniformly convex norm; (2) Every convex body in $X$ can be  approximated by Fr\'echet smooth convex bodies if  $X$ admits an equivalent norm so that its dual norm is locally uniformly convex on $X^*$;   3.  Every convex body in $X$ can be approximated by both locally uniformly convex and  Fr\'{e}chet smooth convex  bodies if  $X$ is reflexive;  4. If $X$ is separable, then   every convex body in $X$ can be  approximated by   both locally uniformly convex and Fr\'{e}chet smooth convex  bodies if and only if $X$ is an Asplund space; (5) the following statements are equivalent: A. $X$ is super reflexive; B. Every convex body in $X$ can be uniformly approximated by uniformly convex  bodies; C. Every convex body in $X$ can be uniformly approximated by uniformly smooth convex  bodies; D.  Every convex body in $X$ can be uniformly approximated by both uniformly convex and uniformly smooth convex  bodies.
\end{abstract}



\maketitle
 \section{Introduction}
   In 1959, after showing that every convex body in a finite dimensional normed space can be approximated by both strictly convex and smooth convex bodies with respect to the Haustorff metric, V. Klee \cite{Klee1959} further asked which infinite dimensional Banach spaces can ensure that similar conclusions hold true.   Recently, Cheng, Jiang and Yuan \cite{CL3} provided some sufficient and necessary conditions for Banach spaces which guarantee that every convex body in it can be approximated by  strictly convex and  G\^{a}teaux smooth convex bodies. In particular,  every convex body in a separable Banach space can be approximated by  both strictly convex and  G\^{a}teaux smooth convex bodies.\\

   In this paper, we continue to study this problem in strong senses, such as local uniform convexity, uniform convexity, Fr\'{e}chet smoothness and uniform smoothness.  As a result, we mainly show the following theorems.
   \begin{theorem} Every convex body in a Banach space $X$ is uniformly approximated by locally uniformly convex bodies if and only if $X$  admits an equivalent locally uniformly convex norm. \end{theorem}

   \begin{theorem} Every convex body in $X$ can be uniformly approximated by strongly smooth convex bodies if  $X$ admits an equivalent norm so that its dual norm is locally uniformly convex on $X^*$.\end{theorem}

   \begin{theorem} If $X$ is a dual space, then  every convex body in $X$ can be uniformly approximated by  Fr\'{e}chet smooth convex  bodies if and only if $X$ is reflexive, which is equivalent to every convex body in $X$ can be uniformly approximated by both locally uniformly convex and Fr\'{e}chet smooth convex  bodies.\end{theorem}

   \begin{theorem} If $X$ is separable, then   every convex body in $X$ can be uniformly approximated by   both locally uniformly convex and Fr\'{e}chet smooth convex  bodies if and only if $X$ is an Asplund space.\end{theorem}

   \begin{theorem} The following statements are equivalent:

   i) $X$ is super reflexive;

   ii) Every convex body in $X$ can be uniformly approximated by uniformly convex  bodies;

   iii) Every convex body in $X$ can be uniformly approximated by uniformly smooth convex  bodies;

   iv)  Every convex body in $X$ can be uniformly approximated by both uniformly convex and uniformly smooth convex  bodies.\end{theorem}

 \section{Preliminaries}

 In this section, we recall key notions and known results, and establish several  properties that will be used in the sequel. The letter $X$ will always be a Banach space and $X^*$ its dual. $B_X$ stands for the closed unit ball of $X$ and $S_X$ for the unit sphere, i.e. $B_X=\{x\in X: \|x\|\leq1\}$ and $S_X=\{x\in X: \|x\|=1\}$. 

  \subsection{Convex body}
A closed bounded convex set $B$ in a real Banach space $X$ is called a \textit{convex body} if its interior $\text{int} B$ is non-empty.


   \begin{proposition}\label{2.3}
    Every Minkowski functional $p_C$ defined on a Banach space $X$ is a nonnegative continuous sublinear functional, i.e.,
    $$p_C(kx)=kp_C(x),\;\;p_C(x+y)\leq p_C(x)+p_C(y),\;\;\forall\;x,y\in\;X\;{\rm and}\;k\geq0.$$

  \end{proposition}

  \subsection{Minkowski functional}
   For any closed convex set $C$ in a Banach space $X$ with $0 \in \text{int}C$, the \textit{Minkowski functional} $p_C$ with respect to (or, generated by) $C$ is given by
   \[
    p_C(x) = \inf\{\lambda > 0 : x \in \lambda C\}\;\;\;\;\mbox{for any $x\in X$}.
   \]
   A convex body $B$ is said to be strictly convex (resp., G\^{a}teaux smooth) if every boundary point of $B$ is an extreme point (resp., $B$ has a unique tangent hyperplane at every boundary point). Let $\mathfrak C(X)$ denote the cone of all nonempty bounded closed convex set in $X$, equipped with the Hausdorff metric $d_H$ defined for $A, B\in\mathfrak C(X)$,
    \begin{equation} d_H(A,B)=\inf\{r>0: A\subset B+B_X, B\subset A+B_X\},
    \end{equation}
   The subcone $\mathfrak{C}_{00}(X)$ consists of all convex bodies in $X$ containing the origin in their interior. Furthermore, we equip the cone $\mathscr{M}_{00}(X)$ of continuous coercive Minkowski functionals on $X$ with the metric:
\begin{equation}
    d(f, g) = \sup_{x \in B_X} |f(x) - g(x)|, \quad \forall f, g \in \mathscr{M}_{00}(X).
\end{equation}

  \subsection{Fenchel's transform}
  An extended real-valued convex function $f$ on a Banach space $X$ is \textit{proper} if $f(x) > -\infty$ for all $x \in X$ and its essential domain $\text{dom}f \equiv \{x \in X : f(x) < +\infty\}$ is non-empty. We say $f$ is \textit{coercive} if $f(x) \to +\infty$ as $\|x\| \to \infty$. The set of all extended real-valued, lower semicontinuous, proper convex functions defined on $X$ is denoted by $\mathscr{C}_{\text{conv}}(X)$.

  \begin{definition}\label{2.5.1}
    For any $f \in \mathscr{C}_{\text{conv}}(X)$, the \textit{Fenchel’s transform} $\mathscr{F}(f): X^* \to \mathbb{R} \cup \{+\infty\}$ is defined by
\begin{equation}\label{2.5.2}
    \mathscr{F}(f)(x^*) = \sup \{ \langle x^*, x \rangle - f(x) : x \in X \}, \quad x^* \in X^*.
\end{equation}
\end{definition}

  \begin{proposition}\label{2.5.3}
   Let $X$ be a Banach space. Then for each  $f\in \mathscr  C_{\rm conv}(X)$, its Fenchel's transform $\mathscr F(f)$ is a $w^*$-lower semicontinuous (proper) convex function.
  \end{proposition}

  We use $\mathscr C^*_{\rm conv}(X^*)$ to denote the collection of all extended real-valued, $w^*$-lower semicontinuous  (hence, lower semicontinuous), proper convex functions defined on $X^*$.

  The following result is classical (see, for instance, \cite{Di}).

  \begin{proposition}\label{2.5.4} 
Let $X = (X, \|\cdot\|)$ be a Banach space and $f = \frac{1}{p}\|\cdot\|^p$ for $p \in (1, +\infty)$. Then $\mathscr{F}(f) = \frac{1}{q}{\|\cdot\|^*}^q$, where $\frac{1}{p} + \frac{1}{q} = 1$ and $\|\cdot\|^*$ is the dual norm on $X^*$. 
  \end{proposition}

  Let $A$ and $B$ be two partially ordered sets. A mapping $M: A\rightarrow B$ is called {\it order preserving} ({resp., \it order-reversing} ),if $a \geq b$ in $A$ implies $M(a) \geq M(b)$ (resp., $M(a) \leq M(b)$). Futhermore, a bijection $M:A\rightarrow B$ is said to be a {\it fully order-presevering} (resp., {\it fully order-reversing}) mapping, if both $M$ and $M^{-1}$ are {\it order preserving} ({resp., \it order-reversing}).

\begin{proposition}\label{2.5.5}
Regarding the Fenchel transform $\mathscr{F}$ on a Banach space $X$, the following properties hold:
\begin{itemize}
    \item[i)] $\mathscr{F} : \mathscr{C}_{\text{conv}}(X) \to \mathscr{C}_{\text{conv}}(X^*)$ is order-reversing;
    \item[ii)] $\mathscr{F} : \mathscr{C}_{\text{conv}}(X) \to \mathscr{C}_{\text{conv}}^*(X^*)$ is fully order-reversing;
    \item[iii)] $\mathscr{F} : \mathscr{C}_{\text{conv}}(X) \to \mathscr{C}_{\text{conv}}(X^*)$ is fully order-reversing if and only if $X$ is reflexive.
\end{itemize}
\end{proposition}

  \subsection{Hausdorff metric}
   Let $\mathfrak{C}(X)$ denote the cone of all convex bodies in a Banach space $X$. We equip $\mathfrak{C}(X)$ with two ``linear'' operations:
   \[
    \alpha B = \{\alpha b : b \in B\}, \quad \alpha \in \mathbb{R}, B \in \mathfrak{C}(X),
   \]
   \[
    A \oplus B = \overline{\{a + b : a \in A, b \in B\}}, \quad A, B \in \mathfrak{C}(X),
   \]
   where $\overline{D}$ represents the norm closure of $D \subset X$. The \textit{Hausdorff metric} $d_H$ on $\mathfrak{C}(X)$ is defined for any $A, B \in \mathfrak{C}(X)$ by
   \[
    d_H(A, B) = \inf\{r > 0 : A \subset B \oplus rB_X, B \subset A \oplus rB_X\}.
   \]
   Unless confusion arises, we simply write $A + B$ for $A \oplus B$. The following property is standard and straightforward to verify.
   
    \begin{proposition}\label{5.1}
     For any Banach space $X$, $\mathfrak C(X)=(\mathfrak C(X),d_H)$ is a complete metric cone.
    \end{proposition}
    
    We define the following subcones of $\mathfrak{C}(X)$:
    \[
    \mathfrak{C}_0(X) = \{B \in \mathfrak{C}(X) : 0 \in B\} \quad \text{and} \quad \mathfrak{C}_{00}(X) = \{B \in \mathfrak{C}(X) : 0 \in \text{int}B\}.
    \]

   \begin{lemma}\label{5.2}
    Assume that $X$ is a Banach space. Then the following hold:

    i) $\mathfrak C_0(X)$ is again a complete cone;

    ii) $\mathfrak C_{00}(X)$ is a dense open cone of $\mathfrak C_0(X)$;

    iii) $\mathfrak C_{00}(X)$ is an open cone of $\mathfrak C(X)$.
   \end{lemma}
 
 For a Banach space $X$, We use $\mathfrak C^*(X^*)$ to denote the cone of all $w^*$-compact convex sets of the dual space $X^*$ endowed with the Hausdorff metric $d_H$. The subcone $\mathfrak{C}_0^*(X^*)$ (resp., $\mathfrak{C}_{00}^*(X^*)$) consists of such sets containing the origin (resp., containing the origin in their interiors).

  Parallel to Lemma \ref{5.2}, we have
  \begin{lemma}\label{5.40}

  i) Both $\mathfrak C(X^*)$ and $\mathfrak C^*_0(X^*)$ are complete cones;

  ii) $\mathfrak C^*_{00}(X^*)$ is a dense open subset of $\mathfrak C^*_0(X^*)$.
  \end{lemma}

 For $p\in [1,+\infty)$, $\mathscr H^p(X)$ stands for the cone of all  positively  homogeneous continuous convex functions of degree $p$ on $X$, endowed with the metric $d$ defined for $f,g\in\mathscr H^p(X)$ by
 \[
 d(f,g)=\sup_{x\in B_X}|f(x)-g(x)|.
 \]
 Furthermore, $\mathscr{H}_c^p(X)$ denotes the subcone of all coercive functions within $\mathscr{H}^p(X)$.

  Analogous to the proof of \cite[Lemma~5.3]{CL3} in the case $p=2$, we derive the following result for genera $p\in[1,+\infty)$.

\begin{lemma}\label{5.3}
Let $X$ be a Banach space and $p\in[1,+\infty)$. Then

i)  $\mathscr H^p(X)$ is a complete metric cone;

ii) given $f\in\mathscr H^p(X)$ and $r>0$, $B_r (\equiv\{x\in X: f(x)\leq r\})$ is a convex body if and only if $f$ is coercive;

iii) $\mathscr H^p_c(X)$ is a dense open subset of $\mathscr H^p(X)$;

iv) given $f\in\mathscr H^p(X)$ and $r>0$,  $B_r$ is a strictly convex body if and only if $f$ is strictly convex and coercive;

v) for each $f\in\mathscr H^p(X)$, there is a closed convex set $C\subset X$ with $0\in{\rm int} C$ so that $f=\frac{1}pp_C^p$;

vi) the convex set $C$ defined in v) is a convex body if and only if $f$ is coercive.
\end{lemma}

With the above notions in place, we formulate the main theorem governing the correspondence between the space of convex bodies and that of their associated Minkowski functionals.

\begin{theorem}\label{5.4}
Suppose that $X$ is a Banach space and $p\in[1,+\infty)$. Then $T:\mathfrak C_{00}(X)\rightarrow\mathscr H^p_c(X)$ defined for $B\in\mathfrak C_{00}(X)$ by  $T(B)=\frac{1}pp^p_B$ is a fully order-reversing locally Lipschitz isomorphism.
\end{theorem}

\begin{proof}
By definition of $T$ and Lemma \ref{5.3}, we see that $T$ is fully order-reversing. It remains to show that $T$ and $T^{-1}$ are both locally Lipschitz.

Next, we show that $T$ is locally Lipschitz.  For any fixed $A\in\mathfrak C_{00}(X)$, then there is $\delta>0$ such that $\delta B_X\subset A$. Put
\[
\mathfrak A_{\delta}=\left\{C\in\mathfrak C(X): d_H(A,C)<\frac{\delta}{2}\right\}.
\]
For any $B\in\mathfrak A_{\delta}$, we see that $\frac{\delta}{2}B_X\subset B$.\footnote{For any $x\notin B$, using the Hahn–Banach separation theorem, there is $\varphi\in X^\ast$ with norm one such that $\sup_{y\in B}\varphi(y)<\varphi(x)$. Note that
\[
\delta\leq\sup_{z\in A}\varphi(z)\leq \sup_{z\in B+\frac{\delta}{2} B_X}\varphi(z)=\sup_{y\in B}\varphi(y)+\frac{\delta}{2}<\varphi(x)+\frac{\delta}{2}.
\]
It follows that $\varphi(x)>\frac{\delta}{2}$, and thus $\Vert x\Vert>\frac{\delta}{2}$. Consequently, $\frac{\delta}{2} B_X\subset B$.} We claim that $T$ is Lipschitz on the neighborhood $\mathfrak A_{\delta}$ of $A$.

For any $B, C\in\mathfrak A_{\delta}$, let $d_H(B,C)=\alpha\delta$, where $\alpha\in [0,1)$. It follows that
\begin{align*}
p_B-p_C&\leq p_B-
p_{B\oplus\alpha\delta B_X}\leq p_B-p_{B\oplus2\alpha B}\\
&=p_B-p_{(1+2\alpha)B}\leq p_B-\frac{1}{1+2\alpha}p_{B}\\
&=\frac{2\alpha}{1+2\alpha} p_B\leq \frac{2\alpha}{1+2\alpha}p_{\frac{\delta}{2}B_X}\leq\frac{4\alpha}{\delta}\Vert\cdot\Vert.
\end{align*}
Similarly, we have $p_C-p_B\leq\frac{4\alpha}{\delta} \Vert\cdot\Vert$.
Consequently,
\[
\vert p_C-p_B\vert\leq\frac{4\alpha}{\delta} \Vert\cdot\Vert=\frac{4}{\delta^2} d_H(B,C)\Vert\cdot\Vert.
\]
Therefore,
\begin{align*}
d(T(B),T(C))&=\sup_{x\in B_X}\vert T(B)(x)-T(C)(x)\vert=\sup_{x\in B_X}\frac{1}{p}|p^p_B(x)-p^p_C(x)|\\
&\leq\sup_{x\in B_X}\max\big\{p^{p-1}_B(x),p^{p-1}_C(x)\big\}\cdot|p_B(x)-p_C(x)|\\
&\leq\sup_{x\in B_X}\left(\frac{2}{\delta}\Vert x\Vert\right)^{p-1}\cdot|p_B(x)-p_C(x)|\\
&\leq \frac{2^{p+1}}{\delta^{p+1}}{d_H(B,C)}.
\end{align*}

Finally, we show that $T^{-1}:\mathscr H^p_c(X)\rightarrow\mathfrak C_{00}(X)$ is again locally Lipschitz. Let $\mathscr M_c(X)$ be the cone of all continuous coercive Minkowski functionals on $X$ endowed with the metric $d_{\mathscr M}$ which is defined for $p,q\in\mathscr M_c(X)$ by
$d_\mathscr M(p,q)=\sup_{x\in B_X}|p(x)-q(x)|$. Then $S: \mathscr H^p_c(X)\rightarrow\mathscr M_c(X)$ defined for $f\in\mathscr H^p_c(X)$ by $S(f)=(pf)^{\frac{1}{p}}$ is Lipschitz. To show $T^{-1}$ is locally Lipschitz, it suffices to prove that the mapping $U: \mathscr M_c(X)\rightarrow \mathfrak C_{00}(X)$ defined for $p_C\in \mathscr M_c(X)$ by $U(p_C)=C$ is locally Lipschitz.

For any fixed $p_A\in \mathscr M_c(X)$, where $A\in\mathfrak C_{00}(X)$, then there is $M>0$ such that $A\subset MB_X$.
We denote by $\mathfrak P_M$ the neighborhood of $A$ given by
\[
\mathfrak P_M=\left\{p_C\in\mathscr M_c(X):C\in\mathfrak C_{00}(X)\;\mbox{and}\;d(p_A,p_C)<\frac{1}{2M}\right\}.
\]
For any $p_B\in\mathfrak P_M$, by positive homogeneity, we have
\[
p_B\geq p_A-\frac{1}{2M}\Vert\cdot\Vert\geq p_{MB_X}-\frac{1}{2M}\Vert\cdot\Vert=\frac{1}{M}\Vert\cdot\Vert-\frac{1}{2M}\Vert\cdot\Vert=\frac{1}{2M}\Vert\cdot\Vert,
\]
and thus $B\subset 2MB_X$.
We claim that $U$ is Lipschitz in the neighborhood $\mathfrak A_{\delta}$ of $A$.

For any $p_B, p_C\in\mathfrak P_M$, let $d_\mathscr M(B,C)=r$
\[
p_C\leq p_B+r\Vert\cdot\Vert\leq p_B+2rM\frac{1}{2M}\Vert\cdot\Vert\leq p_B+2rMp_B=p_{\frac{1}{1+2rM}B}.
\]
This implies $\frac{1}{1+2rM}B\subset C$, and thus
\[
B\subset(1+2rM)C\subset C+4rM^2B_X.
\]
Analogously, $C\subset B+4rM^2B_X$. Consequently,
\[
d_H(B,C)\leq 4rM^2=4M^2d_{\mathscr M}(p_B,p_C).
\]
This shows that the mapping $U: \mathscr M_c(X)\rightarrow \mathfrak C_{00}(X)$ is locally Lipschitz.
\end{proof}

For a subset $A$ of a Banach space $X$, its \textit{polar} set $A^\circ$ is defined as the set
\[
A^\circ=\{x^* \in X^* : \langle x^*, x \rangle \leq 1, \forall x \in A\}.
\]
Correspondingly, the \textit{support function} $\sigma_A$ with respect to $A$ is defined on the dual space $X^*$ by 
\[
\sigma_A(x^*) = \sup_{x \in A} \langle x^*, x \rangle. \quad\quad\forall x^*\in X^*
\]

Given $q \in [1, +\infty)$, let $\mathscr{H}^{*q}(X^*)$ be the cone consisting of  all continuous, $w^*$-lower semicontinuous, positively homogeneous convex functions of degree $q$ on $X^*$. This space is equipped with the uniform metric:
\[
d(f,g)=\sup_{x^*\in X^*,\|x^*\|\leq1}|f(x^*)-g(x^*)|,\;\;\forall f,g\in \mathscr H^{*q}(X^*).
\]
We denote by $\mathscr{H}_c^{*q}(X^*)$ the subcone of all coercive functions within $\mathscr{H}^{*q}(X^*)$.

    Analogously to the proof of \cite[Lemma~6.5]{CL3} in the case $p=q=2$, we derive the following result for general $p,q\in[1,+\infty)$.

    \begin{lemma}\label{5.5}
  Suppose that $X$ is a Banach space and $p,q\in(1,+\infty)$ with $\frac{1}{p}+\frac{1}{q}=1$. Then

   i) for each $f\in \mathscr H^{p}_{c}(X)$ there is a convex body $B\in\mathfrak C_{00}(X)$ such that
     $f=\frac{1}pp_B^p$, where $p_B$ is the Minkowski functional generated by $B$;

   ii) for each $f\in \mathscr H^{p}_{c}(X)$, we have $\mathscr F(f)=\frac{1}qp^q_{B^\circ}=\frac{1}q\sigma^q_{B}$;

   iii) $\mathscr F: \mathscr H^{p}_{c}(X)\rightarrow \mathscr H^{q*}_{c}(X^*)$ is a fully order-reversing isomorphism.
  \end{lemma}
 \section{Locally uniformly convex bodies }

Recall that a Banach space $X$ is said to be  locally uniformly convex (LUC, for short) provided for each fixed $x\in S_X$  and for every sequence $\{x_n\}\subset S_X$ with $\|x+x_n\|\rightarrow 2$, we have $x_n\rightarrow x$. Parallel to this notion, we have the following definition.

 \begin{definition} \label{7.100}
 i)  A convex body $B$ of a Banach space $X$ with $0\in\text{int}(B)$ is called {\it locally uniformly convex} if for each fixed $x\in\text{bd}(B)$  and for every sequence $\{x_n\}\subset \text{bd}(B)$ with $p_B(x+x_n)\rightarrow 2$, we have $x_n\rightarrow x$.

 ii) We say that a convex body $B$ of  $X$ is {\it locally uniformly convex} if there is $x_0\in\text{int}(B)$ such that $B-x_0$ is a locally uniformly convex body.

  \end{definition}
 We denote by $\mathfrak C^{luc}_{00}(X)$ the collection of all locally uniformly convex bodies $B$ with $0\in{\rm int}(B)$. Clearly, the set \[\mathfrak C^{luc}_{00}(X)\equiv X+\mathfrak C^{luc}_{00}(X)\equiv\{x+B: x\in X, B\in\mathfrak C^{luc}_{00}(X)\}\] is the cone of all convex bodies of $X$.

  \begin{proposition}
  Let $B$ be a convex body of a Banach space $X$. Then the following statements are equivalent.

  i) $B$ is locally uniformly convex;

  ii) there is $x_0\in\text{int}(B)$ such that the Minkowski functional $p_{B_0}$ generated by $B_0\equiv B-x_0$ satisfying for each fixed $x\in{\rm bd}(B)$ and  every sequence $\{x_n\}\subset{\rm bd}(B_0)$, we have
\begin{equation}\label{7.101}
    p_{B_0}(x+x_n)\rightarrow 2\;\Longrightarrow\;x_n\rightarrow x.
   \end{equation}

   iii)  for every $x_0\in\text{int}(B)$,  the Minkowski functional $p_{B_0}$ generated by $B_0\equiv B-x_0$ satisfying for each fixed $x\in{\rm bd}(B)$ and  every sequence $\{x_n\}\subset{\rm bd}(B_0)$, we have
\begin{equation}
    p_{B_0}(x+x_n)\rightarrow 2\;\Longrightarrow\;x_n\rightarrow x.
   \end{equation}
   \end{proposition}
 \begin{proof}
 i) $\Longleftrightarrow$ ii) is just Definition \ref{7.100}.  iii) $\Longrightarrow$ ii) is trivial. It remains to show ii) $\Longrightarrow$ iii).
 Let $x_0\in\text{int}(B)$ so that the Minkowski functional $p_{B_0}$ generated by $B_0\equiv B-x_0$ satisfies condition (\ref{7.101}). For any fixed $y_0\in\text{int}(B)$, we claim that the  Minkowski functional $p_{C_0}$ generated by $C_0\equiv B-y_0$ satisfies the following: for each fixed $z\in{\rm bd}(C_0)$ and  every sequence $\{z_n\}\subset{\rm bd}(C_0)$, we have
\begin{equation}
    p_{C_0}(z+z_n)\rightarrow 2\;\Longrightarrow\;z_n\rightarrow z.
   \end{equation}
 Note \[C_0=B-y_0=B-x_0-(y_0-x_0))=B_0-(y_0-x_0)\] and \[{\rm bd}(C_0)={\rm bd}(B_0)-(y_0-x_0).\]
Then \[p_{C_0}(x)=1\Longleftrightarrow p_{B_0}(x+(y_0-x_0))=1.\] Therefore, $x, x_n\in{\rm bd}(C_0)$, and
\[ p_{C_0}(x+x_n)\rightarrow 2\;\Longrightarrow\;p_{B_0}\Big([x+(y_0-x_0)]+[x_n+(y_0-x_0)]\Big)\rightarrow 2.\]
Local uniform convexity of $p_{B_0}$ and $x+(y_0-x_0),x_n+(y_0-x_0)\in{\rm bd}(B_0)$ further entail
\[x_n+(y_0-x_0)\rightarrow x+(y_0-x_0),\;\;{\rm\;equivalently,\;} x_n\rightarrow x.\]
 \end{proof}
 Let $B$ be a nonempty bounded closed convex set of a Banach space $X$. A point $x\in B$ is said to be a strongly exposed point of $B$ provided there is $x^*\in X^*$ such that for every $\eps>0$ there is $\delta>0$ such that the diameter $\text{diam}S(B,x^*,\delta)<\eps$, where  $S(B,x^*,\delta)$ is the slice of $B$ determined by $x^*$ and $\delta$, that is, $S(B,x^*,\delta)=\{x\in B:\langle x^*,x\rangle>\langle x^*,x_0\rangle-\delta\}$. In this case, $x^*$ is called a strongly exposing functional of $B$ and strongly exposes $B$ in $x_0$. It is easy to observe that every boundary point of a locally uniformly convex body $B$ is a strongly exposed point of $B$.
The following property states that the converse is also true for a Fr\'echet smooth convex body.

 \begin{proposition}
 Let $B$ be a Fr\'echet smooth convex body of a Banach space $X$. Then $B$ is locally uniformly convex if and only if every point in $\text{bd}(B)$ is a strongly exposed point of $B$.
 \end{proposition}
  \begin{proof}
  Since the necessity immediately follows from definitions of locally uniformly convex body and strongly exposed point,
  it suffices to show sufficiency.

 Suppose that $B$ is a Fr\'echet smooth convex body satisfying that every point in the boundary $\text{bd}(B)$ is a strongly exposed point of $B$. Without loss of generality, we can assume that $0\in\text{int} B$.

  Let $p_B$ be the Minkowski functional generated by $B$, and let $\{x_n\}\subset\text{bd}(B)$ so that $p_B(x+x_n)\rightarrow 2$.  This entails that for all $t>0$,  $p_{B}(tx+x_n)\rightarrow t+1$. Therefore, there is a sequence $\{t_n\}$ with $t_n\downarrow 0^+$ such that \[
  \frac{p_B(t_nx+x_n)-1}{t_n}\rightarrow 1.
  \]
  For each $n\in\mathbb N$, let $x^*_n=d_Fp_B(t_nx+x_n)$, the Fr\'echet derivative of $p_B$ at $t_nx+x_n$. Then
  \begin{equation}\label{7.102}
  1\geq\langle x_n^*,x\rangle\geq\frac{\langle x_n^*,t_nx+x_n\rangle-1}{t_n}=\frac{p_B(t_nx+x_n)-1}{t_n}\rightarrow 1.
  \end{equation}
  (\ref{7.102}) and Fr\'echet smoothness of $B$ entail $x^*_n\rightarrow x^*\equiv d_Fp_B(x)$. Let $z^*$ be a strongly exposing functional which is strongly exposing $B$ at $x$ with $\langle z^*,x\rangle=1$. This, incorporating of Fr\'echet smoothness of $B$ imply that $z^*=x^*$. Thus, $x^*$ is  strongly exposing $B$ at $x$. Since $\langle x^*_n,x_n\rangle\rightarrow 1$ and since $x^*_n\rightarrow x^*$,  we see $\langle x^*,x_n\rangle\rightarrow 1$.
  Since $x^*$ strongly exposes $B$ at $x$, it follows that $x_n\rightarrow x$. Hence $B$ is locally uniformly convex.
 \end{proof}

  We denote by $\mathfrak C^{luc}_{00}(X)$ the collection of all locally uniformly convex bodies $B$ with $0\in{\rm int}(B)$. For general convex bodies, a convex body $B\in\mathfrak C(X)$ is defined to be {\it locally uniformly convex} if $B\in X+\mathfrak C^{luc}_{00}(X)$;  equivalently, $B-z\in \mathfrak C^{luc}_{00}(X)$ for some $z\in{\rm int}(B)$.
  We use $\mathfrak{C}^{luc}(X)$ (resp.,$\mathfrak C_0^{luc}(X)$)  to denote the set of all locally uniformly convex bodies in $X$ (resp., all locally uniformly convex bodies in $X$ that contain the origin). It is straightforward to verify that every locally uniformly convex body is strictly convex; however, the converse fails to hold in general.

  \begin{example}
   We define  a new norm $\||\cdot\||$ on for $(x(n))_{n=1}^\infty\in\ell_\infty$ by
   $\||x\||=\|x\|+\sqrt{\sum_{n=1}^\infty2^{-n}x^2(n)}$, where $\|\cdot\|$ is the original norm of $\ell_\infty$. Then $\ell_\infty$ is strictly convex with respect to the new norm, which is equivalent to that $B_{X,\||\cdot\||}$ is a strictly convex body. But it is not locally uniformly convex.
  \end{example}

   Let $C$ be a convex set of a Banach space $X$, and $f$ be a convex function defined on $C$.
   The function $f$ is said to be {\it locally uniformly convex} if for each $x_0\in C$ and for every $\eps>0$, there exists $\delta>0$ such that
   \[x\in C,\;\|x-x_0\|\geq\eps\;\Longrightarrow \frac{1}2[f(x_0)+f(x)]-f\Big(\frac{x_0+x}2\Big)>\delta.\]

  \begin{remark}
   A locally uniformly convex function is necessarily strictly convex; however, the converse does not hold in general: a strictly convex function need not be locally uniformly convex. It is also worth noting that a locally uniformly convex norm on a Banach space does not qualify as a locally uniformly convex function. Nevertheless, the following result holds.
  \end{remark}

  \begin{proposition}\label{7.202}
   A Banach space $(X,\|\cdot\|)$ is locally uniformly convex if and only if $f=\|\cdot\|^2$ is a locally uniformly convex function.
  \end{proposition}
   \begin{proof}
    {\bf Sufficiency}. Suppose that $f=\|\cdot\|^2$ is a locally uniformly convex function. Let $x_0\in S_X$ and $\{x_n\}\subset S_X$ such that $\|x_0+x_n\|\rightarrow2$. If $x_n\nrightarrow x$, then there exist a subsequence $\{z_n\}$ of $\{x_n\}$ and $\eps>0$ such that $\|z_n-x_0\|\geq\eps$ for all $n\in\N$. It follows that \[f\Big(\frac{x_0+z_n}2\Big)=\Big\|\frac{x_0+z_n}2\Big\|^2\rightarrow1=\frac{1}2\Big(f(x_0)+f(z_n)\Big),\] which contradicts to the local uniform convexity of $f$.

    {\bf Necessity}. Assume that $(X,\|\cdot\|)$ is locally uniformly convex. If $f$ were not locally uniformly convex, then there would exist $x_j\in X,\;j=0,1,2,\cdots$ and $\eps>0$ such that  $\|x_n-x_0\|\geq\eps$ and
    \begin{equation}\label{7.203}
     \frac{1}2(f(x_0)+f(x_n))- f\Big(\frac{x_0+x_n}2\Big)\rightarrow0.
    \end{equation}
   This implies
     \begin{equation}
     0\leftarrow\frac{1}2(\|x_0\|^2+\|x_n\|^2)- \frac{1}4\|x_0+x_n\|^2\geq\frac{1}4(\|x_n\|-\|x_0\|)^2,
     \end{equation}
   so $\|x_n\|\rightarrow\|x_0\|$. Without loss of generality, we may assume that $\|x_n\|=1=\|x_0\|$ for all $n\in\N$. Then it follows from (\ref{7.203}) that $\|x_n+x_0\|\rightarrow2$, which contradicts the local uniform convexity of $(X,\|\cdot\|)$.
   \end{proof}

  The following propertyis a direct consequence of the definition of locally uniformly convex functions.
  \begin{proposition}\label{7.205}
   Let $f,g$ be continuous convex functions on a Banach space $X$. If one of $f$ and $g$ is locally uniformly convex, then $f+g$ is also locally uniformly convex.
  \end{proposition}
  For a Banach space $X$, we use $\mathscr H^2_{luc}(X)$ (resp.,$\mathscr H^2_{lucc}(X)$) to denote the set of all continuous positive quadratic homogeneous and locally uniformly convex (resp., continuous positive quadratic homogeneous coercive and locally uniformly  convex) functions on $X$.

  \begin{lemma}\label{7.3}
   Let $f\in\mathscr H^2_c(X)$ and $r>0$. Then  $B_r=\{x\in X: f(x)\leq r\}$ is a locally uniformly convex body if and only if $f$ is a locally uniformly convex  function.
  \end{lemma}
   \begin{proof}
    Since $f$ is continuous, positive quadratic homogeneous,  coercive, and convex, we may assume without loss of generality that $r=\frac{1}2$.Then $B\equiv\{x\in X: \sqrt{2f(x)}\leq1\}$ is a convex body satisfying $f=\frac{1}2p_B^2$, where $p_B$ is the minkowski functional generated by $B$.

    {\bf Sufficiency}. Suppose that $f=\frac{1}2p_B^2$ is a locally uniformly convex function. Let $x_0\in{\rm bd}(B)$ and $\{x_n\}\subset{\rm bd}(B)$ satisfy $p_B(x_0+x_n)\rightarrow2$. If $x_n\nrightarrow x$, then there exist a subsequence $\{z_n\}$ of $\{x_n\}$ and $\eps>0$ such that $\|z_n-x_0\|\geq\eps$ for all $n\in\N$. It follows that
     \[
      f\Big(\frac{x_0+z_n}2\Big)=\frac{1}2p_B\Big(\frac{x_0+z_n}2\Big)^2\rightarrow\frac{1}2=\frac{1}2(f(x_0)+f(z_n)),
     \]
     which  contradicts the local uniform convexity of $f$.

    {\bf Necessity}. Suppose that $B$ is a locally uniformly convex body. We aimto show that $f=\frac{1}2p_B^2$ is a locally uniformly convex function. Assume for contradiction that $f$ is not locally uniformly convex, then there exist $x_j: j=0,1,2,\cdots$ and $\eps>0$ such that $\|x_n-x_0\|\geq\eps,\;n\in\N$ and
    \[
    \frac{1}2(f(x_0)+f(z_n))-f\Big(\frac{x_0+x_n}2\Big)\rightarrow0.
    \]
    This implies
    \[
    0\leftarrow 2(p_B^2(x_0)+p_B^2(x_n))-p_B^2(x_0+x_n)\geq(p_B(x_0)-p_B(x_n))^2\geq0,
    \]
    so $p_B(x_0)>0$ and $p_B(x_n)\rightarrow p_B(x_0)$. Without loss of generality, we can assume $p_B(x_n)=1=p_B(x_0)$ for all $n\in\N$.
    Then $p_B(x_0+x_n)\rightarrow2,$
    which contradicts the local uniform convexity of $B$.
   \end{proof}

  \begin{lemma}\label{7.4}
   Let $X$ be a Banach space admitting an equivalent locally uniformly convex norm. Then the set $\mathscr H^2_{lucc}(X)$  of all continuous positive quadratic homogeneous locally uniformly convex coercive functions contains a dense subset of  $\mathscr H^2(X)$.
  \end{lemma}
   \begin{proof}
    Without loss of generality, assume the original norm $\|\cdot\|$ of $X$ is locally uniformly convex. For each $n\in\N$, define $f_n=\frac{1}{2n}\|\cdot\|^2$. Then $f_n\in\mathscr H^2_{lucc}(X)$.  By Proposition \ref{7.205}, $f_n+\mathscr H^2(X)\subset\mathscr H^2_{lucc}(X)$.
    Clearly, $\bigcup_{n=1}^\infty(f_n+\mathscr H^2(X))$ is a dense  subset of  $\mathscr H^2(X)$.
   \end{proof}
  To establish the main result of this section, we leverage the conclusion of Lemma \ref{7.4} to prove the following key theorem.
  \begin{theorem}\label{7.5}
   Let $X$ be a Banach space admitting an equivalent locally uniformly convex norm. Then $\mathfrak{C}_{00}^{luc}(X)$ contains a dense  subset of $\mathfrak C_0(X)$.
  \end{theorem}
   \begin{proof}
    By Theorem \ref{5.4},  $T: \mathfrak C_{00}(X)\rightarrow \mathscr H^2_c(X)$ defined by $T(B)=\frac{1}2p^2_B$ is a fully order-reversing continuous isomorphism. Thus, $T^{-1}:\mathscr H^2_c(X) \rightarrow \mathfrak C_{00}(X)$ is also a fully order-reversing continuous isomorphism.
    Since $\mathscr H^2_{lucc}(X)$  contains a dense  subset of  $\mathscr H^2(X)$ (see, Lemma \ref{7.4}), it necessarily contains a dense subset of $\mathscr H_c^2(X)$. Consequently, $T^{-1}\big(\mathscr H^2_{lucc}(X)\big)$ contains a dense  subset of $\mathfrak C_{0}(X)$ (Lemma \ref{5.2} ii)). We complete the proof by noting that each element in $T^{-1}\big(\mathscr H^2_{lucc}(X)\big)$ is a locally uniformly convex body (Lemma \ref{7.3}).
   \end{proof}
 \begin{theorem}\label{7.6}
  In a Banach space $X$, the set $\mathfrak{C}^{luc}(X)$ consisting of all locally uniform convex bodies of $X$  is dense in $\mathfrak C(X)$ if and only if $X$ admits an locally uniform convex norm.
 \end{theorem}
  \begin{proof}
   {\bf Necessity}.  Assume that $\mathfrak{C}^{luc}(X)$  is dense in $\mathfrak C(X)$. Take any $B\in \mathfrak C^{luc}(X)$ with $0\in\text{int}B$. Since $B$ is locally uniform convex, by Proposition \ref{7.202}, $p_B^2$ is locally uniform convex, where $p_B$ is the Minkowski functional generated by $B$. Then  $\||\cdot\||:=\sqrt{p_B^2(\cdot)+{p_{-B}^2(\cdot)}}$ is an equivalent locally uniform convex norm of $X$.

   {\bf Sufficiency}. Suppose $X$ admits an equivalent locally uniform convex norm. Then by Theorem \ref{7.5}, the set $\mathfrak{C}_{00}^{luc}(X)$ contains a dense  subset of $\mathfrak C_{0}(X)$. We complete the proof by noting $\mathfrak C(X)=X+\mathfrak C_{00}(X)$ and $X+\mathfrak C_{00}^{luc}(X)=\mathfrak C^{lucc}(X)$.
  \end{proof}
  By Theorem \ref{7.6} and Troyanski renorming theorem, i.e. every separable Banach space admits an equivalent locally uniformly convex norm, we have the following consequence.
  \begin{corollary}
   For every separable Banach space $X$, the set $\mathfrak{C}^{luc}(X)$ of all locally uniformly convex bodies of $X$ is always dense in $\mathfrak C(X)$.
  \end{corollary}
\begin{remark}
Since every separable Banach space can be renormed so that $X$ is locally uniformly convex and $X^*$ is strictly convex,  by Asplund averaging technique \cite{Asp} and a procedure which is similar to the proof of \cite[Th.7.1]{CL3}, we can further show that ``every convex body in a separable Banach space can be approximated by both locally uniformly convex and G\^ateaux smooth convex bodies".
\end{remark}

 \section{Fr\'echet smooth convex bodies}
    A Banach space $X=(X,\|\cdot\|)$ is said to be Fr\'echet smooth  (FS, for short) provided the norm $\|\cdot\|$ is everywhere Fr\'echet differentiable  off the origin. This property is equivalent to the condition that for each $x\in S_X$, there exists $x^*\in X^*$ such that
   \[\limsup_{t\rightarrow0, y\in B_X}\Big(\frac{\|x+ty\|-1}t-\langle x^*,y\rangle\Big)=0.\]
  Parallel to this definition, we have the following definition.
  \begin{definition}
 Let $X$ be a Banach space.

  i) A convex body $B\in\mathfrak C_{00}(X)$ is said to be {\it Fr\'echet smooth } if the Minkowski functional $p_B$ generated by $B$ is everywhere Fr\'echet differentiable  off the origin, that is,  for each  $x\in{\rm bd}(B)$, there exists $x^*\in X^*$ such that
  \begin{equation}\label{8.101}
   \limsup_{t\rightarrow0, y\in B_X}\Big(\frac{p_B(x+ty)-1}t-\langle x^*,y\rangle\Big)=0.
  \end{equation}
  We denote by  $\mathfrak C^{fs}_{00}(X)$ the set of all  Fr\'echet smooth convex bodies $B$ with $0\in{\rm int}(B)$.

  ii) We say that a (general) convex body $B\in\mathfrak C(X)$ is {\it Fr\'echet smooth } provided $B\in X+\mathfrak C^{fs}_{00}(X)$, which is equivalent to that there is $x_0\in\text{int}(B)$ so that $B-x_0\in \mathfrak C^{fs}_{00}(X)$. We use $\mathfrak{C}^{fs}(X)$  to denote the set of all Fr\'echet smooth  convex bodies in $X$.
  \end{definition}
 The following result is parallel to Proposition 3.2, and its proof is also similar to the proof of Proposition 3.2.
 \begin{proposition} Let $B$ be a convex body of a Banach space $X$. Then the following statement are equivalent.

  i) $B$ is Fr\'echet smooth;

  ii) there is $x_0\in\text{int}(B)$ such that the Minkowski functional $p_{B_0}$ generated by $B_0\equiv B-x_0$ is everywhere Fr\'echet differentiable off the origin;

   iii)  for every $x_0\in\text{int}(B)$,  the Minkowski functional $p_{B_0}$ generated by $B_0\equiv B-x_0$ is everywhere Fr\'echet differentiable off the origin.

   \end{proposition}

 Clearly, a Fr\'echet smooth convex body is a G\^ateaux smooth convex body, but the converse is not true.

  \begin{example}\label{8.102}
   Let $X\in\{\ell_1, C[0,1]\}$. The set $\mathfrak C^{\rm gsm}(X)$ of all G\^ateaux smooth convex bodies contains a dense  subset of  $\mathfrak{C}(X)$. But $\mathfrak{C}^{fs}(X)=\emptyset.$
   Indeed, since $X$ is separable, it admits an equivalent norm so that $X^*$ is strictly convex with respect to the new norm. By\cite[Theorem~6.9]{CL3}, $\mathfrak C^{\rm gsm}(X)$ contains a dense  subset of   $\mathfrak{C}(X)$.
   If there exists $B\in\mathfrak{C}^{fs}(X)$, then for each $z\in{\rm int}(B)$, $C\equiv B-z\in\mathfrak{C}_{00}^{fs}(X)$. Let $\||\cdot\||=p_C+p_C(-\cdot)$. It is not difficult to see that $\||\cdot\||$ is an equivalent norm which is Fr\'echet differentiable everywhere off the origin. This is a contradiction (see, for instance, \cite{EL}).
  \end{example}

  The following property is classical (see, instance, \cite{Ph}).

  \begin{proposition}\label{8.2}
   Let $X$ be a Banach space and $f:X\rightarrow\R$ be a continuous convex function. Then the following statements are equivalent.

   i) $f$ is  Fr\'echet (G\^{a}teaux, resp.) differentiable at $x_0\in X$;

   ii) the subdifferential mapping $\partial f: X\rightarrow2^{X^*}$ is single-valued and norm-to-norm (norm-to-$w^*$, resp.) up semicontinuous at $x_0$;

   iii) every selection of $\partial f$ is norm-to-norm (norm-to-$w^*$, resp.) continuous at at $x_0$;

   iv) there exists a selection of $\partial f$ which is  norm-to-norm (norm-to-$w^*$, resp.) continuous at at $x_0$.
  \end{proposition}
  \begin{definition}
  Let $C$ be a bounded convex set of a Banach space $X$ and $x_0\in C$.

  i) The point $x_0$ is said to be an {\it exposed point} of $C$ provided that there exists $x_0^*\in X^*$ such that
  \[
  \langle x^*_0,x_0\rangle>\langle x^*,x\rangle>,\;\forall\;x\in C\setminus\{x_0\}.
  \]
  In this case, $x^*_0$ is called an {\it exposing functional} of $C$ and exposing  $C$ at $x_0$;

  ii) $x_0$ is called  a {\it strongly exposed point} if there is $x_0^*\in X^*$ such that for every sequence $\{x_n\}\subset C$

  \[
  \langle x^*,x_n\rangle\rightarrow\langle x^*,x_0\rangle\;\Longrightarrow \;x_n\rightarrow x_0.
  \]
   In this case, $x^*_0$ is called a {\it strongly exposing functional} of $C$ and strongly exposing  $C$ at $x_0$.

  iii) If, in addition, $C$ is $w^*$ closed in the dual $X^*$, then we say that the point $x_0\in C$ is a {\it $w^*$-exposed point}( resp., {\it $w^*$-exposed point}) of $C$ if it is a  exposed (resp.,{\it $w^*$-strongly exposed point})  point of $C$ and the exposing (resp., strongly exposing) functional $x^*$ is coming from $X$, instead of $X^{**}$.
  In this case, $x^*$ is called a {\it $w^*$-exposing functional} (resp., {\it $w^*$-strongly exposing functional}) of $C$ and exposing (resp., {$w^*$-strongly exposing})  $C$ at $x_0$.
 \end{definition}
 The following property is classical.
   \begin{proposition}\label{8.203}
  Let $p_B$ be the Minkowski functional generated by a convex body $B$ of a Banach space $X$, and let $x_0\in X$. Then the following statements are equivalent.

  i) $p_B$ is Fr\'echet (G\^ateaux, resp.) differentiable at $x_0$ with its Fr\'echet (G\^ateaux, resp.) derivative $x_0^*$;

  ii) $x_0^*$ is a $w^*$-strongly exposed ($w^*$-exposed, resp.) point of the polar $B^\circ$ of $B$ and $x_0$ is  strongly exposing ($w^*$-exposing, resp.) $B^\circ$  at $x_0^*$.
   \end{proposition}

  \begin{lemma}\label{8.3}
   Let $X$ be a Banach space.If $g\in\mathscr H^{*2}_{c}(X^*)$ is locally uniformly convex, then $f\equiv\mathscr F^{-1}(g)\in \mathscr H^{2}_{c}(X)$ is  Fr\'echet differentiable everywhere.
  \end{lemma}
   \begin{proof}
    Let $g\in\mathscr H^{*2}_{c}(X^*)$, i.e., $g$ is a continuous and $w^*$-lower semicontinuous positive quadratic homogeneous coercive convex function on $X^*$. By Lemma \ref{5.5} ii), there exists a convex body $B\in\mathfrak C_{00}(X)$ with $B^\circ\in\mathfrak C^*_{00}(X^*)$ such that
    $g=\frac{1}2\sigma_B^2=\frac{1}2p_{B^\circ}^2$.  Therefore, $f\equiv\mathscr F^{-1}(g)=\frac{1}2p_{B}^2$.

    It follows from \cite[Lemma~6.7]{CL3} that $f$ is everywhere G\^{a}teaux differentiable, which is equivalent to that the subdifferential mapping $\partial f$ is single-valued and norm-to-$w^*$ continuous at each point of $X$. Suppose, to the contrary, that there exists $x_0\in X$ such that $f$ is not Fr\'echet differentiable at $x_0$. Clearly, $x_0\neq0$. Positive quadratic homogeneity of $f$ allows us to assume that $p_B(x_0)=1$. By Proposition \ref{8.2}, there exist $\eps>0$ and a sequence $\{x_n\}$ with $p_B(x_n)=1$ and with $\;x_n\rightarrow x_0$ such that $\|x^*_n-x_0^*\|\geq\eps$, where $x^*_j=d_G f(x_j),\;j=0,1,2,\cdots$, the G\^ateaux derivative of $f$ at $x_j$. Since $p_B(x_n)=1=\langle x^*_n,x_n\rangle,\;n=0,1,2,\cdots$, and since
    $x_n\rightarrow x_0$, we obtain
    \[
    g(\frac{x_0^*+x_n^*}2)=\frac{1}2\sigma_B^2(\frac{x_0^*+x_n^*}2)\geq\frac{1}2\langle\frac{x_0^*+x_n^*}2,x_0\rangle^2\;\;\;\;\;\;
    \]
    \[
    \rightarrow\frac{1}2=\frac{1}2\Big(\frac{1}2\sigma_B^2(x_0^*)+\frac{1}2\sigma_B^2(x_n^*)\Big)=\frac{1}2(g(x_0^*)+g(x_n^*).
    \]
    Since $\|x^*_n-x_0^*\|\geq\eps$, this contradicts to  locally uniform convexity of $g$.
   \end{proof}
  The following lemma is a key link to prove the main result of this section.
  \begin{lemma}\label{8.4}
   Suppose that $X$ is a Banach space admitting an equivalent norm so that $X^*$ is locally uniformly convex.  Then the set $\mathscr H^{*2}_{lucc}(X^*)$  of all continuous and $w^*$-lower semicontinuous quadratic homogenous locally uniformly convex coercive functions contains a dense  subset of  $\mathscr H^{*2}(X^*)$, the set of all continuous and $w^*$-lower semicontinuous quadratic homogenous  convex  functions.
  \end{lemma}
   \begin{proof}
    Without loss of generality, assume that the original norm $\|\cdot\|_*$ of $X^*$ is $w^*$-lower semicontinuous locally uniformly convex. For each $n\in\N$, let $g_n=\frac{1}{2n}\|\cdot\|_*^2$. Then
    \[
    g_n+\mathscr H^{*2}(X^*)\subset\mathscr H^{*2}_{lucc}(X^*),\;n=1,2,\cdots.
    \]
    Clearly,
    $\bigcup_{n=1}^\infty(g_n+\mathscr H^{*2}(X^*))$ is a dense subset of $\mathscr H^{*2}(X^*)$.
   \end{proof}
  Now, we are ready to prove the following main theorem of this section.
  \begin{theorem}\label{8.5}
   Suppose that $X$ is a Banach space admitting an equivalent norm so that   $X^*$ is locally uniformly convex with respect to the new  norm. Then the set $\mathfrak C^{fs}(X)$  of all Fr\'echet smooth convex bodies contains a dense  subset of $\mathfrak C(X)$.
  \end{theorem}
   \begin{proof}
    By Lemma \ref{5.2} and \ref{5.40}, we know that both $\mathfrak C(X^*)$ and $\mathfrak C^*_0(X^*)$ are complete cones, and $\mathfrak C^*_{00}(X^*)$ is a dense open set of $\mathfrak C^*_0(X^*)$.
    Lemma \ref{5.5} iii) implies  the Fenchel transform $\mathscr F: \mathscr H^{2}_{c}(X)\rightarrow \mathscr H^{2*}_{c}(X^*)$ is a fully order-reversing isomorphism. By Theorem \ref{5.4}, the map $T:\mathfrak C_{00}(X)\rightarrow\mathscr H^2_c(X)$ defined by  $T(B)=\frac{1}2p^2_B$ is a fully order-reversing continuous isomorphism, where $p_B$ is the Minkowski functional generated by $B$. Therefore,
    \[
     U\equiv T\circ\mathscr F^{-1}:\mathscr H^{2*}_{c}(X)\rightarrow\mathfrak C_{00}(X)
    \]
    is an order-preserving continuous isomorphism.

    By Lemma \ref{8.4}, $\mathscr H^{*2}_{lucc}(X^*)$ contains a dense  subset of  $\mathscr H_c^{*2}(X^*)$. Consequently, $\frak W\equiv U\big(\mathscr H^{*2}_{lucc}(X^*)\big)$ contains a dense  subset of $\mathfrak C_{0}(X)$. Since $\frak C_{0}(X)$ is  a dense open subset of $\mathfrak C(X)$, $\mathfrak W\subset\mathfrak C^{fs}_{00}(X) $ contains a dense  subset of $\mathfrak C_{0}(X)$. Since every element of $\mathfrak W$ is a Fr\'echet smooth convex body, we complete the proof  by noting that $X+\mathfrak W$ contains a dense  subset of $X+\mathfrak C_0(X)=\mathfrak C(X)$.
   \end{proof}
  \begin{theorem}\label{8.6}
   If $X$ is a reflexive Banach space, then the set of all locally uniformly convex and Fr\'echet smooth convex bodies, i.e. $\mathfrak C^{luc}(X)\bigcap\mathfrak C^{fs}(X)$ contains a dense  subset of $\mathfrak C(X)$.
  \end{theorem}
   \begin{proof}
   Since $X$ is reflexive, it has an equivalent norm so that both $X$ and its dual $X^*$ are locally uniformly convex. Without loss of generality, we assume that the original norm $\|\cdot\|$ of $X$ has such a property, i.e. both $\|\cdot\|$ and its dual norm $\|\cdot\|^*$ are locally uniformly convex. Fix any  $B\in\mathfrak C_{00}(X)$. We can claim that for every $\eps>0$, there exist a locally uniformly convex body $A\in\mathfrak C_{00}(X)$ contained in $B$ and a  convex body $C\in\mathfrak C_{00}(X)$ containing $B$ so that its polar $C^\circ=\{x^*\in X^*: \langle x^*,x\rangle\leq1;\;\forall x\in C\}$ is a $w^*$-closed locally uniformly convex body of $X^*$ such that  $d_{H}(A,C)<\eps.$ Indeed, let $p_A, p_B$ and $p_C$ be successively, the Minkowski functionals generated by $A, B$ and $C$. Then for any $\delta>0$, both $f_\delta\equiv\delta\|\cdot\|^2+p_A^2$ and $g_\delta^*\equiv\delta\|\cdot\|{^*}^2+p_{C^{\circ}}^2$ are locally uniformly convex, and $g_\delta^*$ is $w^*$-lower semicontinuous on $X^*$.
Put $A_\delta=\{x\in X: f_\delta(x)\leq1\}$ and $C_\delta=\{x^*\in X^*:g_\delta^*(x^*)\leq1\}^\circ$. Clearly, $A_\delta\subset B\subset C_\delta$ and
$d_{H}(A_\delta,C_\delta)<\eps$ for all sufficiently small $\delta>0$.

Starting with $f_0=\frac{1}2p_A^2, g_0=\frac{1}2p_C^2$,\;let \[f_n=\frac{1}2(f_{n-1}+g_{n-1}), \;g_n=\mathscr F^{-1}\left(\frac{1}2\mathscr F(f_{n-1})+\frac{1}2\mathscr F(g_{n-1})\right),\; n=1,2,\cdots,\]
where $\mathscr F$ is the Fenchel transform.Then for all $n\in\N$, \[f_0\geq f_1\geq\cdots\geq f_{n-1}\geq f_n\geq g_n\geq g_{n-1}\geq\cdots\geq g_0,\]
and
\[\mathscr F(f_0)\leq \mathscr F(f_1)\leq\cdots\leq \mathscr F(f_{n-1})\leq\mathscr F(f_n)\leq \mathscr F(g_n)\leq \mathscr F(g_{n-1})\leq\cdots\leq \mathscr F(g_0).\]
 Note that $f_0$ is locally uniformly convex on $X$  and $\mathscr F(g_0)=\frac{1}2p_{C^*}^2$ is locally uniformly convex on $X^*$, and that $\mathscr F(g_n)=\frac{1}2(\mathscr F(f_{n-1})+\mathscr F(g_{n-1}))$. Then for each $n\in\mathbb N$, $f_n$ is locally uniformly convex on $X$ and $\mathscr F(g_n)$ is locally uniformly convex on $X^*$. Monotonicity of both the sequences $\{f_n\}$ and $\{g_n\}$ and $\sup_{x\in B_X}|f_n(x)-g_n(x)|\rightarrow 0$ entail that that there is a convex body $D\in\mathfrak C_{00}$ such that
 \[\lim_nf_n=\frac{1}2p_D^2=\lim_ng_n,\;\;\text{uniformly on}\;B_X.\]
 Clearly, $A\subset D\subset C$. Since $A\subset B\subset C$ with $d_{H}(A,C)<\eps$, we get $d(B,D)<\eps$. By an argument as the same as the proof of Theorem 1 in \cite{Asp}, we see that $\frac{1}2p_D^2$ is locally uniformly convex on $X$, and $\mathscr F(\frac{1}2p_D^2)$ is locally uniformly convex on $X^*$. It follows that $D$ is both locally uniformly convex and Fr\'echet smooth.
   \end{proof}
 Recall that a Banach space $X$ is an Asplund space provided that every continuous convex function on $X$ is Fr\'echet differentiable at each point of a dense $G_\delta$ subset of $X$, which is equivalent to that the dual of each separable subspace of $X$ is again separable.
  \begin{theorem}
   Suppose that $X$ is a separable Banach space.  Then every convex body can be approximated by both locally uniformly convex and Fr\'echet smooth convex bodies, or, equivalently, $\mathfrak C^{luc}(X)\bigcap C^{fs}(X) $ contains a dense  subset of $\mathfrak C(X)$, if and only if $X$ is an Asplund space.
  \end{theorem}
   \begin{proof} {\bf Sufficiency}. Since $X$ is a separable Asplund space, $X^*$ is also separable. Therefore, $X$ admits an equivalent locally uniformly convex norm so that its dual is again locally uniformly convex. Thus, the sufficiency follows from Theorem \ref{8.5}.

   {\bf Necessity}. Since every convex body can be approximated by Fr\'echet smooth convex bodies, there is a Fr\'echet smooth convex body $B\in\mathfrak C_{00}(X).$ Let $p_{\pm B}$ be the Minkowski functionals generated by $\pm B$.
   Then both $p^2_{ B}$ and $p^2_{ -B}$ are everywhere Fr\'echet differentiable in $X$. Put $\||\cdot\||=\sqrt{p^2_{ B}+p^2_{ -B}}$. Then it is easy to see that $\||\cdot\||$ is an equivalent Fr\'echet smooth norm on $X$. This implies that the dual $X^*$ is separable. Indeed, let $\||\cdot\||$ be an equivalent Fr\'echet smooth norm on $X$. Then the subdifferential mapping $\partial\||\cdot\||: X\rightarrow 2^{X^*}$ is single-valued and norm-to-norm continuous at each nonzero point of $X$. In particular, $\partial\||\cdot\||: S_X\rightarrow S_{X^*}$ is norm-to-norm continuous. Note that $x^*\in \partial\||\cdot\||(S_X)$ if and only if $x^*$ is a norm-attaining functional. By the Bishop-Phelps theorem, norm-attaining functionals are dense in $S_{X^*}$. Therefore, the set $\partial\||\cdot\||(S_X)$ is dense in $S_{X^*}$. Since $\partial\||\cdot\||$ is continuous on $S_X$, $S_{X^*}$ is separable.  Consequently, $X$ is an Asplund space.
   \end{proof}
   Note that $c_0$ is a separable Asplund space, the following result follows.
   \begin{corollary}
   Every convex body of $c_0$ can be approximated by both locally uniformly convex and Fr\'echet smooth convex bodies.
   \end{corollary}

 \section{Uniformly convex bodies}
  Recall that a Banach space $X=(X,\|\cdot\|)$ is said to be  uniformly convex (UC, for short) provided  for any two sequences $\{x_n\}, \{y_n\}\subset S_X$, $\|x_n+y_n\|\rightarrow 2$ implies  $x_n-y_n\rightarrow 0$. Analogously, we have the following definition.
  \begin{definition} Let $X$ be a Banach space.

  i) A convex body $B\in\mathfrak C_{00}(X)$ is is said to be   {\it uniformly convex} if  for any  sequences $\{x_n\},\{y_n\}\subset{\rm bd}(B)$, the implication
  \begin{equation}\label{9.101}
   p_B(x_n+y_n)\rightarrow 2\;\Longrightarrow\;x_n-y_n\rightarrow 0
  \end{equation}
  holds, where $p_B$ denotes the Minkowski functional of $B$. We denote by $\mathfrak C^{uc}_{00}(X)$ the collection of all  uniformly convex bodies $B$ with $0\in{\rm int}(B)$.

  ii) A convex body $B$ of $X$ is called {\it uniformly convex} if $B-x_0\in\mathfrak C^{uc}_{00}(X)$ for some $x_0\in\text{int}(B)$. We use $\mathfrak{C}^{uc}(X)$   to denote the set of all  uniformly convex bodies in $X$.
  \end{definition}
  The following property is easy to follow.
  \begin{proposition}
  Suppose that $B$ is a convex body of a Banach space. Then the following are equivalent.

  i) $B$ is a uniformly convex body;

  ii) for every $x_0\in\text{int}(B)$, the Minkowski functional $p$ generated by $B_0\equiv B-x_0$ satisfies that for any  sequences $\{x_n\},\{y_n\}\subset{\rm bd}(B_0)$,
  \[p(x_n+y_n)\rightarrow 2\;\Longrightarrow\;x_n-y_n\rightarrow\;0.\]
  \end{proposition}
  Recall that a convex function defined on a convex set $C$ of a Banach space $X$ is said to be uniformly convex provided for every $\eps>0$, there exists $\delta>0$ such that
  \[x, y\in C, \;\|x-y\|\geq\eps\;\Longrightarrow\;\frac{1}2(f(x)+f(y))-f(\frac{x+y}2)>\delta.\]
  \begin{proposition}\label{9.202}
   A Banach space $(X,\|\cdot\|)$ is  uniformly convex if and only if $f=\|\cdot\|^2$ is a  uniformly convex function.
  \end{proposition}
   \begin{proof}
    {\bf Sufficiency}. Suppose that $f=\|\cdot\|^2$ is uniformly convex. Let  $\{x_n\},\{y_n\}\subset S_X$ satisfying $\|x_n+y_n\|\rightarrow2$. If $x_n-y_n\nrightarrow 0$, then there exist subsequences $\{u_n\}\subset\{x_n\}$ , $\{v_n\}\subset\{y_n\}$ and $\eps>0$ such that $\|u_n-v_n\|\geq\eps$ for all $n\in\N$. It follows that
    \[
    f(\frac{u_n+v_n}2)=\|\frac{u_n+v_n}2\|^2\rightarrow1=\frac{1}2(f(u_n)+f(v_n)),
    \]
    which contradicts to the uniform convexity of  $f$.

    {\bf Necessity}. Assume that $(X,\|\cdot\|)$ is  uniformly convex. Suppose that $f$ is not uniformly convex. Then there  exist $x_n, y_n\in X,\;n=1,2,\cdots$ and $\eps>0$ such that $\|x_n-y_n\|\geq\eps$ and
    \begin{equation}\label{9.203}
    \frac{1}2(f(x_n)+f(y_n))- f(\frac{x_n+y_n}2)\rightarrow0.
    \end{equation}
    This implies
    \begin{equation}\label{9.204}0\leftarrow\frac{1}2(\|x_n\|^2+\|y_n\|^2)- \frac{1}4\|x_n+y_n\|^2\geq\frac{1}4(\|x_n\|-\|y_n\|)^2,
    \end{equation}
    so $\|x_n\|-\|x_0\|\rightarrow0$. Without loss of generality, we may assume $\|x_n\|=1=\|y_n\|$ for all $n\in\N$. Then (\ref{9.203}) implies $\|x_n+y_n\|\rightarrow2$, which contradicts to the uniform convexity of   $(X,\|\cdot\|)$.
    \end{proof}
 The following property is easy to obtain.
 \begin{proposition}
 Let $B$ be a convex body of a Banach space $X$. Then the following statements are equivalent.

 i) $B$ is uniformly convex;

 ii) there is $x_0\in\text{int}(B)$ such that $B-x_0\in\mathfrak C^{uc}_{00}(X) $;

 iii) for every $x_0\in\text{int}(B)$, $B-x_0\in\mathfrak C^{uc}_{00}(X) $.
 \end{proposition}

  Clearly, every uniformly convex body is locally uniformly convex, but  converse is not true. The following property immediately follows  from the definition.
  \begin{proposition}\label{9.102}
   A Banach space  is uniformly convex if and only if its closed unit ball is a uniformly convex body.
  \end{proposition}
 Recall that a Banach space $X$ is said to be super reflexive if every Banach space $Y$ which is finitely representable in $X$ is reflexive, that is,
  for every $\eps>0$ and for every finite dimensional subspace $Y_0$ of $Y$, there exist a finite dimensional subspace $X_0$ of $X$ and a linear isomorphism $T: Y_0\rightarrow X_0$ such that $\|T\|\|T^{-1}\|<1+\eps$.

  \begin{proposition}\label{9.102b}
   A Banach space $X$ admits a uniformly convex body if and only if it is super reflexive.
  \end{proposition}
   \begin{proof}
    {\bf Sufficiency}.  It follows from Enflo's renorming theorem \cite{En} (see, also \cite{Pi}) that every super reflexive Banach space admits an equivalent norm so that it is uniformly convex with respect to the new norm. Let $B$ be the closed unit ball with respect to the new norm. Then $B$ is a uniformly convex body.

    {\bf Necessity}. Let $B$ be a uniformly convex body of $X$ with $0\in{\rm int}(B)$. By definition, $p_B(x_n+y_n)\rightarrow2$ implies $\|x_n-y_n\|\rightarrow0$ whenever $\{x_n\},\;\{y_n\}\subset{\rm bd}(B)$. Define $\||\cdot\||=p_B+p_B(-\cdot)$, it is straightforward to verify that $\||\cdot\||$ is an equivalent uniformly convex norm on $X$. By James' theorem \cite{Ja}, $X$ is super reflexive.
   \end{proof}

  The following property follows directly from the definition of uniformly convex functions, mirroring the corresponding result for locally uniformly convex functions.
  \begin{proposition}\label{9.205}
   Let $f,g$ be continuous convex functions on a Banach space $X$. If one of $f$ and $g$ is  uniformly convex, then $f+g$ is again uniformly convex.
  \end{proposition}
  For a Banach space $X$, we use $\mathscr H^2_{uc}(X)$ (resp.,$\mathscr H^2_{ucc}(X)$) to denote the set of all continuous positive quadratic homogeneous and uniformly convex (resp., continuous positive quadratic homogeneous coercive and  uniformly  convex) functions on $X$.

   The following is a key lemma linking uniform convexity of a convex function to  its level sets.
  \begin{lemma}\label{9.3}
   Let $f\in\mathscr H^2_c(X)$, and $r>0$. Then  $B_r\equiv\{x\in X: f(x)\leq r\}$ is a  uniformly convex body if and only if $f$ is a  uniformly convex  function.
  \end{lemma}
  \begin{proof}
   Since $f$ is continuous convex positively quadratic homogeneous and coercive, we can assume  $r=\frac{1}2$.  This entails that $B\equiv\{x\in X: \sqrt{2f(x)}\leq1\}$ is a convex body satisfying $f=\frac{1}2p_B^2$, where $p_B$ is the Minkowski functional generated by $B$.

   {\bf Sufficiency}. Suppose $f=\frac{1}2p_B^2$ is a  uniformly convex function. Let  $\{x_n\},\{y_n\}\subset{\rm bd}(B)$ with $p_B(x_n+y_n)\rightarrow2$. If $x_n-y_n\nrightarrow 0$, then there exist  subsequence $\{u_n\}\subset\{x_n\}$, $\{v_n\}\subset\{y_n\}$ and $\eps>0$ such that $\|u_n-v_n\|\geq\eps$ for all $n\in\N$. It follows that
   \[
   f(\frac{u_n+v_n}2)=\frac{1}2p_B(\frac{u_n+v_n}2)^2\rightarrow\frac{1}2=\frac{1}2(f(u_n)+f(v_n)),
   \]
   which contradicts to the uniform convexity of $f$.

   {\bf Necessity}. Suppose that $B$ is a locally uniformly convex body. We want to show that $f=\frac{1}2p_B^2$ is a  uniformly convex function.  Assume  that $f$ is not uniformly convex. Then there exist $x_n, y_n\in X (n=1,2,\cdots)$ and $\eps>0$ such that $\|x_n-y_n\|\geq\eps$ and
   \[
   \frac{1}2(f(x_n)+f(y_n))-f(\frac{x_n+y_n}2)\rightarrow0.
   \]
   This implies
   \[
   0\leftarrow 2(p_B^2(x_n)+p_B^2(y_n))-p_B^2(x_n+y_n)\geq(p_B(x_n)-p_B(y_n))^2\geq0.
   \]
   so $p_B(x_n),p_B(y_n)\nrightarrow0$ and $p_B(x_n)-p_B(y_n)\rightarrow 0$. Without loss of generality, we can assume $p_B(x_n)=1=p_B(y_n)$ for all $n\in\N$. Then $x_n,y_n\in{\rm bd}(B)$ and $p_B(x_n+y_n)\rightarrow2$,
   which contradicts to the uniform convexity of  $B$.
   \end{proof}
  \begin{lemma}\label{9.4}
   Let $X$ be a Banach space admitting an equivalent  uniformly convex norm. Then the set $\mathscr H^2_{ucc}(X)$  of all continuous quadratic homogenous  uniformly convex coercive functions contains a dense  subset of  $\mathscr H^2(X)$.
  \end{lemma}
   \begin{proof}
    Without loss of generality, assume that the original norm $\|\cdot\|$ of $X$ is  uniformly convex. For each $n\in\N$, let $f_n=\frac{1}{2n}\|\cdot\|^2$. Then $f_n\in\mathscr H^2_{ucc}(X)$. By Proposition \ref{9.205},
     $f_n+\mathscr H^2(X)\subset\mathscr H^2_{ucc}(X)$. Since $f_n\rightarrow 0$ uniformly on $B_X$, it follows that
    \[
    \bigcup_{n=1}^\infty(f_n+\mathscr H^2(X))\subset\mathscr H^2_{ucc}(X)
    \]
    is a dense  subset of  $\mathscr H^2(X)$.
   \end{proof}
  \begin{theorem}\label{9.5}
   Let $X$ be a Banach space admitting an equivalent uniformly convex norm. Then $\mathfrak{C}_{00}^{uc}(X)$ contains a dense  subset of $\mathfrak C_0(X)$.
  \end{theorem}
   \begin{proof}
    By Theorem \ref{5.4},  $T: \mathfrak C_{00}(X)\rightarrow \mathscr H^2_c(X)$ defined by $T(B)=\frac{1}2p^2_B$ is a fully order-reversing continuous isomorphism. Thus, $T^{-1}:\mathscr H^2_c(X) \rightarrow \mathfrak C_{00}(X)$ is also a fully order-reversing continuous isomorphism. By Lemma \ref{9.4}, $\mathscr H^2_{ucc}(X)$  contains a dense  subset of  $\mathscr H^2(X)$.  Note  that $\mathscr H^2_{ucc}(X)\subset\mathscr H^2_{c}(X)$. Then   $T^{-1}\big(\mathscr H^2_{ucc}(X)\big)$ contains a dense subset of $\mathfrak C_{00}(X)$. Since  $\mathfrak C_{00}(X)$ is a dense open subset of $\mathfrak C_{0}(X)$ (Lemma \ref{5.2} ii)),  $T^{-1}\big(\mathscr H^2_{ucc}(X)\big)$ contains a dense  subset of $\mathfrak C_{0}(X)$. We complete the proof by noting that each element in $T^{-1}\big(\mathscr H^2_{ucc}(X)\big)$ is a  uniformly convex body (Lemma \ref{9.3}).
   \end{proof}
  \begin{corollary}\label{9.6}
   Let $X$ be a Banach space admitting an equivalent  uniformly convex norm. Then $\mathfrak{C}^{uc}(X)$ contains a dense  subset of $\mathfrak C(X)$.
  \end{corollary}
   \begin{proof}
    By Theorem \ref{9.5}, it suffices to note $\mathfrak{C}^{ucc}(X)=X+\mathfrak{C}_{00}^{ucc}(X)$ and $\mathfrak{C}(X)=X+\mathfrak{C}_{00}(X)$.
   \end{proof}
  Combining  Corollary \ref{9.6} with the characterization of super reflexive spaces  (Proposition  \ref{9.102}), we have the following.
  \begin{theorem}\label{9.7}
   Let $X$ be a Banach space. Then the following statements are equivalent.

   i) $X$ is super reflexive;

   ii) $X$ has a uniformly convex body;

   iii) $\mathfrak{C}^{uc}(X)$ contains a dense  subset of $\mathfrak C(X)$.
  \end{theorem}
 \section{Uniformly smooth convex bodies}
  Recall that a Banach space $X=(X,\|\cdot\|)$ is said to be uniformly smooth (US, for short) provided the Fr\'echet derivative $d_F\|\cdot\|$ of the norm $\|\cdot\|$ is uniformly continuous on the unit sphere $S_X$ of $X$, which is  equivalent to
  \[
   \limsup_{t\rightarrow0, x, y\in S_X}\Big(\frac{\|x+ty\|+\|x-ty\|-2}t\Big)=0.
  \]
  Parallel to this definition, we have the following.

  \begin{definition} Let $X$ be a Banach space.

  i) A convex body $B\in\mathfrak C_{00}(X)$ is said to be {\it  uniformly smooth convex} provided the Fr\'echet derivative $d_Fp_B$ of the Minkowski functional $p_B$ generated by $B$ is uniformly continuous on the boundary  ${\rm bd}(B)$, or equivalently,
  \begin{equation}\label{10.101}
   \limsup_{t\rightarrow0, x,y\in{\rm bd}(B)}\Big[\frac{p_B(x+ty)+p_B(x-ty)-2}t\Big]=0.
  \end{equation}
  We denote by $\mathfrak C^{us}_{00}(X)$ the set of all uniformly smooth convex bodies $B$ contained in $\mathfrak C_{00}(X)$.

  ii) We say that a convex body $B\in\mathfrak C(X)$ is {\it  uniformly smooth convex} if there is $x_0\in\text{int}(B)$ such that $B-x_0\in\mathfrak C^{us}_{00}(X)$.
  We use $\mathfrak{C}^{us}(X)$ to denote set of all uniformly  smooth  convex bodies in $X$
\end{definition}

It is clear that every uniformly  smooth convex body is  Fr\'echet smooth, but the converse does not hold in general. The following property is easy to follow.
\begin{proposition}
Suppose that $X$ is a Banach space, and $B\in\mathfrak C(X)$ is a convex body. Then the following are equivalent.

i) $B$ is a uniformly  smooth convex body;

ii) For every $x_0\in\text{int}(B)$, the Minkowski functional $p$  generated by $B-x_0$ is uniformly Fr\'echet differentiable on $\text{bd}(B)$.
\end{proposition}

  The following property is classical.
  \begin{proposition}\label{10.2}
   Let $X$ be a Banach space, $f:X\rightarrow\R$ be a continuous convex function, and $D\subset X$ be a bounded set. Then the following statements are equivalent:

   i) $f$ is uniformly Fr\'echet differentiable  on $D$;

   ii) the subdifferential mapping $\partial f: X\rightarrow2^{X^*}$ is single-valued and norm-to-norm uniformly continuous on $D$;

   iii) every selection of $\partial f$ is norm-to-norm uniformly continuous on $D$;

   iv) there is a selection of $\partial f$ which is  uniformly continuous on $D$.
  \end{proposition}
The following result is not difficult to check.
  \begin{proposition}\label{10.203}
  Let $X$ be a Banach space,  $B\in\mathfrak C_{00}(X)$ be a convex body, and  let $p_B$ be the Minkowski functional generated by  $B$. Then the following statements are equivalent.

  i) $p_B$ is uniformly  Fr\'echet differentiable on the boundary ${\rm bd}(B)$ of $B$;

  ii) The Fr\'echet derivative $d_Fp_B$ is uniformly continuous on  ${\rm bd}(B)$.
  \end{proposition}
  The following  lemma is a connection between  uniform convexity of a convex functions $f$ on the dual space $X^*$ of $X$ and  uniform Fr\'echet differentiability of its Fenchel transform $\mathscr (f)$ on $X$.
  \begin{lemma}\label{10.3}
   Let $X$ be a Banach space.
   If $g\in\mathscr H^{*2}_{c}(X^*)$ is uniformly convex then $f\equiv\mathscr F^{-1}(g)\in \mathscr H^{2}_{c}(X)$ is uniformly Fr\'echet differentiable on each bounded subset of $X$.

  \end{lemma}
   \begin{proof}
    Let $g\in\mathscr H^{*2}_{c}(X^*)$, i.e. $g$ is a continuous and $w^*$-lower semicontinuous positive quadratic homogeneous coercive convex function on $X^*$. By Lemma \ref{5.5} ii), there exists a convex body $B\in\mathfrak C_{00}(X)$ with $B^\circ\in\mathfrak C^*_{00}(X^*)$ such that
    $g=\frac{1}2\sigma_B^2=\frac{1}2p_{B^\circ}^2$, and such that $f\equiv\mathscr F^{-1}(g)=\frac{1}2p_{B}^2$.

    By Lemma \ref{8.3}, $f$ is everywhere Fr\'echet differentiable. Suppose, to the contrary, that $f$ is not uniformly Fr\'echet differentiable on some bounded set $D\subset X$.  Positive quadratic homogeneity of $f$ allows us to assume $D={\rm bd}(B)$.  By Proposition \ref{10.203}, there exist $\eps>0$ and two sequences $\{x_n\}, \{y_n\}$ with $p_B(x_n)=1=p_B(y_n), \;x_n-y_n\rightarrow 0$ such that $\|x^*_n-y_n^*\|\geq\eps$, where $x^*_n=d_Ff(x_n),\;y^*_n=d_Ff(y_n),\;n=1,2,\cdots$. Since
    \[
    \langle x^*_n,x_n\rangle=p_B(x_n)=1=p_B(y_n)=\langle y^*_n,y_n\rangle,\;n=1,2,\cdots,
    \]
    and since $x_n-y_n\rightarrow 0$, we obtain
    \[
     g(\frac{x_n^*+y_n^*}2)=\frac{1}2\sigma_B^2(\frac{x_n^*+y_n^*}2)\geq\frac{1}2\langle\frac{x_n^*+y_n^*}2,x_n\rangle^2\;\;\;\;\;\;
    \]
    \[
    \rightarrow\frac{1}2=\frac{1}2\Big(\frac{1}2\sigma_B^2(x_n^*)+\frac{1}2\sigma_B^2(y_n^*)\Big)=\frac{1}2(g(x_n^*)+g(y_n^*).
    \]
Since $\|x^*_n-y_n^*\|\geq\eps$,  this is a  contradiction to the uniform convexity of $g$.

   \end{proof}
  \begin{lemma}\label{10.4}
   Suppose that $X$ is a Banach space and  $X^*$ has an equivalent  uniformly convex norm. Then the set $\mathscr H^{2}_{ucc}(X^*)$  of all continuous  quadratic homogeneous  uniformly convex coercive functions contains a dense subset of  $\mathscr H^{2}(X^*)$.
  \end{lemma}
   \begin{proof}
    Since $X^*$ has an equivalent  uniformly convex norm, $X^*$ (hence, $X$) is super reflexive. Therefore, the $w^*$-topology of $X$ (resp., $X^*$) coincides with the weak topology of $X$ (resp. $X^*$). Since every continuous convex function is weak-lower semicontinuous, we obtain $\mathscr H^{*2}(X^*)=\mathscr H^{2}(X^*)$.

    Without loss of generality, we assume that the dual norm $\|\cdot\|_*$ of $X^*$ is uniformly convex. For each $n\in\N$, let $g_n=\frac{1}{2n}\|\cdot\|_*^2$. Then $g_n\in\mathscr H^{2}_{ucc}(X^*)$  for all $n\in\N$. We have
    \[
    g_n+\mathscr H^{2}(X^*)\subset\mathscr H^{2}_{ucc}(X^*),\;n=1,2,\cdots.
    \]
    Therefore, $\mathscr H^{2}_{ucc}(X^*)$ contains the dense  subset
    $\bigcup_{n=1}^\infty\Big(g_n+\mathscr H^{2}(X^*)\Big)$
    of $\mathscr H^{2}(X^*)$.
   \end{proof}
   We will use Lemma \ref{10.4} to prove the main theorem regarding  density of uniformly smooth convex bodies.
  \begin{theorem}\label{10.5}
   Suppose that $X$ is a Banach space and that $X^*$ has an equivalent uniformly convex  norm. Then the set $\mathfrak C^{us}(X)$  of all uniformly smooth convex bodies  contains a dense  subset of $\mathfrak C(X)$.
  \end{theorem}
   \begin{proof}
    Since  $X^*$ has an equivalent uniformly convex  norm, $X$ is reflexive. Therefore, $\mathfrak C^*(X^*)=\mathfrak C(X^*)$.
    By Lemma \ref{5.2} and \ref{5.40}, we know that both $\mathfrak C(X^*)$ and $\mathfrak C^*_0(X^*)=\mathfrak C_0(X^*)$ are complete cones; and $\mathfrak C^*_{00}(X^*)=\mathfrak C_{00}(X^*)$ is a dense open set of $\mathfrak C_0(X^*)$.
    Lemma \ref{5.5} iii) implies that  $\mathscr F: \mathscr H^{2}_{c}(X)\rightarrow \mathscr H^{2}_{c}(X^*)$ is a fully order-reversing isomorphism. By Theorem \ref{5.4}, the map $T:\mathfrak C_{00}(X)\rightarrow\mathscr H^2_c(X)$ defined for $B\in\mathfrak C_{00}(X)$ by  $T(B)=\frac{1}2p^2_B$ is a fully order-reversing continuous isomorphism. Therefore,
    \[
    U\equiv T\circ\mathscr F^{-1}:\mathscr H^{2}_{c}(X^*)\rightarrow\mathfrak C_{00}(X)
    \]
    is an order-preserving continuous isomorphism.

    By Lemma \ref{10.4}, $\mathscr H^{2}_{ucc}(X^*)$ contains a dense  subset of  $\mathscr H_c^{2}(X^*)$. Consequently, $\frak W\equiv U\big(\mathscr H^{2}_{ucc}(X^*)\big)$ contains a dense  subset of $\mathfrak C_{00}(X)$. Since $\mathfrak C_{00}(X)$ is  a dense open subset of $\mathfrak C(X)$, $\mathfrak W\subset\mathfrak C^{us}_{00}(X) $ contains a dense  subset of $\mathfrak C_{0}(X)$. Since every element of $\mathfrak W$ is a uniformly smooth convex body, we finish the proof by noting that $X+\mathfrak W$ contains a dense  subset of $X+\mathfrak C_0(X)=\mathfrak C(X)$.
   \end{proof}
   By some known renorming theorems, we can show the following.
  \begin{theorem} Let $X$ be a Banach space. Then the following are equivalent.

  i) $X$ is super reflexive;

  ii) $\mathfrak C^{us}(X)$ is dense  in $\mathfrak C(X)$;

  iii)  $\mathfrak C^{us}(X)$ is dense in $\mathfrak C(X)$;

  iv) $\mathfrak C^{uc}(X)\bigcap\mathfrak C^{uc}(X) $ is dense $\mathfrak C(X)$.
  \end{theorem}
  \begin{proof}
 By Enflo's renorming theorem,   every super reflexive Banach space $X$ can be renormed to be uniformly convex. Equivalently, $X^*$  can be renormed to be uniformly smoth. Note that every uniformly convex Banach space is superreflevive. Then  i) \;$\Longleftrightarrow$ ii) follows from Theorem \ref{10.5}. i) \;$\Longleftrightarrow$ \;iii) is  Theorem \ref{9.7}. It remains to show i\;)$\Longrightarrow$ iv).

 Since $X$ is super reflexive, it can be renormed so that both $X$ and $X^*$ are uniformly convex. By the procedure of the proof of Theorem \ref{8.6}, we finish the proof simply  by substituting the uniform convexity of the norm on $X$ and the dual norm on $X^*$ for locally uniform convexity of the norm on $X$ and the dual norm on $X^*$.
  \end{proof}

 \section{Uniformly convex  bodies of power type }
  Recall that for a Banach space $X=(X,\|\cdot\|)$, its modulus of convexity  $\delta_X: [0,2)\rightarrow[0,1]$ is defined for  $\eps\in[0,2)$ by
  \begin{equation}\label{11.1}
  \delta_X(\eps)=\inf\{1-\frac{\|x+y\|}2: x,y\in X, \|x\|=1=\|y\|,\;\|x-y\|\geq\eps\}.
  \end{equation}
  It is well known that $X$ is uniformly convex if and only if $\delta_X(\eps)>0$ for all $0<\eps<2$. It is worth to mention that  Enflo \cite{En} showed that every super reflexive Banach space $X$ can be renormed to be uniformly convex in  1972,. In 1975, Pisier \cite{Pi} further  showed that every super reflexive Banach space $X$ can be renormed so that its modulus of convexity has power type $p$, that is, there exist $2\leq p<+\infty$ and a constant $c>0$ such that $\delta_X(\eps)>c\eps^p$ for all $0<\eps<2$. Thus, the Enflo-Pisier renorming theorem can be restated as ``every super reflexive Banach space can be renormed so that it is uniformly convex with a convexity modules  of power  type $p$ for some $2\leq p<+\infty$. In this case, we simply say that  $X$ is  {\it p-uniform convex}.

  \begin{definition}\label{11.2} Let $X$ be a Banach space and $p\in[2,+\infty)$.
   We say a convex body $B\subset X$ is {\it p-uniform convex} provided  there is $x_0\in{\rm int}(B)$ such that the Minkowski functional $p_0$ generated by $B_0\equiv B-x_0$ satisfies
  \begin{equation}\label{11.201}
   \delta_{B_0}(\eps)=\inf\left\{1-p_0(\frac{x+y}2): \;x,y\in{\rm bd}(B_0),\|x-y\|\geq\eps\right\}\geq c\eps^p,\;\eps\in[0,2),
  \end{equation}
for some constant $c>0$.
  \end{definition}
  \begin{lemma}\label{11.202}
  Let $X$ be a Banach space and $p\in[2,+\infty)$.  A convex body $B\subset X$ is  p-uniform convex if and only if for every $z_0\in{\rm int}(B)$, there is a constant $c>0$ such that
   \[\delta_{B_0}(\eps)=\inf\left\{1-p_0(\frac{x+y}2): \;x,y\in{\rm bd}(B_0),\|x-y\|\geq\eps\right\}\geq c\eps^p,\;\eps\in[0,2),\]
  where $p_0$ is the Minkowski functional generated by $B-z_0$.
  \end{lemma}
  \begin{proof}
  By definition of $p$-uniform convex body, it suffices to show necessity.
  Let  $x_0\in{\rm int}(B)$ and $c>0$  such that the Minkowski functional $p_0$ generated by $B_0\equiv B-x_0$ satisfies
  \begin{equation}\label{11.203}
   \delta_{B_0}(\eps)=\inf\left\{1-p_0(\frac{x+y}2): \;x,y\in{\rm bd}(B_0),\|x-y\|\geq\eps\right\}\geq c\eps^p,\;\eps\in[0,2).
  \end{equation}
Fix any $z_0\in\text{int}(B)$. Let $C=B-z_0$, and $p_C$ be the Minkowski functional generated by $C$. We want to show
 \begin{equation}\label{11.204}
   \delta_{C}(\eps)=\inf\left\{1-p_C(\frac{x+y}2): \;x,y\in{\rm bd}(C),\|x-y\|\geq\eps\right\}\geq c_1\eps^p,\;\eps\in[0,2).
  \end{equation}
  Note that $C=B-z_0=B-x_0-(z_0-x_0)$. Then $\text{bd}(C)=\text{bd}(B_0)-(z_0-x_0)$.  Therefore, $p_C(x)=1\Longleftrightarrow p_0\Big(x+(z_0-x_0)\Big)=1$. Equivalently, \[p_C(x)=p_0\Big(x+p_C(x)\cdot(z_0-x_0)\Big), \;\forall\;x\in X.\]\;
  Consequently, $x,y\in C$ implies  $\frac{x+y}2+p_C(\frac{x+y}2)\cdot(z_0-x_0)\in B_0$,
  \[p_C(\frac{x+y}2)=p_0\Big(\frac{x+y}2+p_C(\frac{x+y}2)\cdot(z_0-x_0)\Big)\]
  \[=p_0\Big(\frac{x+p_C(\frac{x+y}2)\cdot(z_0-x_0)}2+\frac{y+p_C(\frac{x+y}2)\cdot(z_0-x_0)}2)\Big),\]
  and \[\Big\|\frac{x+p_C(\frac{x+y}2)\cdot(z_0-x_0)}2-\frac{y+p_C(\frac{x+y}2)\cdot(z_0-x_0)}2\Big\|=\frac{1}2\|x-y\|.\]
  For $\eps\in[0,2)$,
   \[ \delta_{C}(\eps)=\inf\left\{1-p_C(\frac{x+y}2): \;x,y\in{\rm bd}(C),\|x-y\|\geq\eps\right\}\]
   \[=\inf\left\{1-p_C(\frac{x+y}2): \;x,y\in C,\|x-y\|\geq\eps\right\}\]
   \[=\inf\left\{1-p_0\Big(\frac{x+p_C(\frac{x+y}2)\cdot(z_0-x_0)}2+\frac{y+p_C(\frac{x+y}2)\cdot(z_0-x_0)}2)\Big): \;x,y\in C,\|x-y\|\geq\eps\right\}\]
   \[\geq\inf\left\{1-p_0\Big(\frac{u+v}2\Big): \;u,v\in B_0,\|u-v\|\geq\frac{\eps}2\right\}\geq\frac{1}{2^p}c\eps^p.\]
   Therefore, we finish the proof by taking $c_1=\frac{1}{2^p}c$.
  \end{proof}
  For each $p\in [2,+\infty)$, we denote by $\mathfrak{C}^{uc\text{-}p}(X)$ the set of  $p$- uniformly convex bodies, and by $\mathfrak C^{uc\text{-}p}_{00}(X)\subset\mathfrak{C}^{uc\text{-}p}(X)$ containing all $B\in\mathfrak{C}^{uc\text{-}p}(X)$ with $0\in\text{int}(B)$. The following theorem which follows from Lemma \ref{11.202} says that $p$-uniformly convex body is translation invariant.
  \begin{theorem}
  Let $X$ be a Banach space, and $2\leq p<+\infty$. Then
  \[\mathfrak{C}^{uc\text{-}p}(X)=X+\mathfrak C^{uc\text{-}p}_{00}(X).\]
  \end{theorem}

  The next characterization of $p$-uniformly convex bodies can be found in \cite[Proposition 4.15]{Ch2009}.

  \begin{theorem}\label{12.2}
   Let $B$ be a convex body  of a Banach space $X$ with $0\in{\rm int}(B)$. Then $B$ is $p$-uniformly convex if and only if there exists $c>0$, such that $\forall x,y\in X$,
   \begin{equation}\label{12.2.1}
    p^p_B(\frac{x+y}{2})\leq\frac{p^p_B(x)+p^p_B(y)}{2}-cp_B^p(\frac{x-y}{2}).
   \end{equation}
  \end{theorem}
  Let $C$ be a convex set of a Banach space $X$, and $f$ be a convex function defined on $C$. The function $f$ is said to be  {\it $p$-uniformly convex}  if  there is $c>0$ such that for  all $\eps>0$,
  \[
  x,y\in C,\;\|x-y\|\geq\eps\;\Longrightarrow \frac{1}2[f(x)+f(y)]-f(\frac{x+y}2)\geq c\eps^p.
  \]
  \begin{remark}
   It is worth to mention that  the norm of a $p$-uniformly convex Banach space is not a $p$-uniformly convex function. Nevertheless, by Theorem \ref{12.2}, we have the following.
  \end{remark}
  \begin{proposition}
   Let $p\in[2,+\infty)$. A Banach space $(X,\|\cdot\|)$ is  {\it p-uniformly convex} if and only if $f=\|\cdot\|^p$ is a $p$-uniformly convex function.
  \end{proposition}
  The following property follows directly from  definition of $p$-uniformly convex function.
  \begin{proposition}\label{12.3}
   Let $p\in[2,+\infty)$, and $f,g$ be continuous convex functions on a Banach space $X$. If one of $f$ and $g$ is $p$-uniformly convex, then $f+g$ is again $p$-uniformly convex.
  \end{proposition}

  Let $p\in[1,+\infty)$. Recall that  $\mathscr H^p(X)$ denotes the cone of all  positively  homogeneous continuous  convex functions of degree $p$ on $X$, endowed with the metric $d$ defined for $f,g\in\mathscr H(X)$ by
  $d(f,g)=\sup_{x\in B_X}|f(x)-g(x)|$, and $\mathscr H^p_c(X)$ stands for the subcone of $\mathscr H^p(X)$ consisting of all  positively homogeneous continuous coercive convex functions of degree $p$ on $X$.


  \begin{theorem}\label{12.9}
   For every super reflexive Banach space $X$, there exits $p\in[2,+\infty)$ such that the set $\mathfrak C^{uc\text{-}p}(X)$ of all $p$- uniformly convex bodies  contains a dense  subset of $\mathfrak C(X)$.
  \end{theorem}
   \begin{proof}
    By Enflo-Pisier's renorming theorem for super reflexive spaces, without loss of generality, we can assume  that $X$ is a $p$-uniformly convex Banach space for some $p\in[2,+\infty)$. Denote by  $\mathscr H^p_{c}(X)$ the set of all positively $p$-th homogeneous, continuous, and coercive convex functions, and by $\mathscr H^p_{uc,p}(X)$ the set of all continuous, positively $p$-th homogeneous, $p$-uniformly convex, coercive convex functions. For each $n\in\N$, let $f_n=\frac{1}{pn}\|\cdot\|^p$. Then $f_n\in\mathscr H^p_{uc,p}(X)$. By Proposition \ref{12.3},
    $f_n+\mathscr H^p_c(X)\subset\mathscr H^p_{uc,p}(X)$.  Since $f_n\rightarrow 0$ uniformly on $B_X$, $\bigcup{n=1}^\infty\Big(f_n+\mathscr H^p_c(X)\Big)$ is a dense subset of  $\mathscr H^p_c(X)$.

    By Theorem \ref{5.4}, the map $T_p:\mathfrak C_{00}(X)\rightarrow\mathscr H^p_{c}(X)$ defined  for $B\in\mathfrak C_{00}(X)$ by $T_p(B)=\frac{1}{p}p^p_B$  is a fully order-reversing continuous isomorphism. Thus, $T_p^{-1}:\mathscr H^p_c(X) \rightarrow \mathfrak C_{00}(X)$ is also a fully order-reversing continuous isomorphism.

    Since $\mathscr H^p_{uc,p}(X)$  contains a dense  subset of  $\mathscr H^p_c(X)$, it follows that $T_p^{-1}\big(\mathscr H^p_{uc,p}(X)\big)$ contains a dense  subset of $\mathfrak C_{00}(X)$. Since  $\mathfrak C_{00}(X)$ is a dense open subset of $\mathfrak C_{0}(X)$ (Lemma \ref{5.2} ii)),  $T_p^{-1}\big(\mathscr H^p_{uc,p}(X)\big)$ contains a dense  subset of $\mathfrak C_{0}(X)$. We complete the proof by noting that every element in $T_p^{-1}\big(\mathscr H^p_{uc,p}(X)\big)$ is a $p$- uniformly convex body  (Theorem \ref{12.2}).
   \end{proof}
  As an immediate consequence of Theorem \ref{12.9}, we have the following.
  \begin{corollary}
   If $X$ is isomorphic to a $L_p$ space for some $1<p<\infty$,  then the set $\mathfrak C^{uc\text{,}q}(X)$ containing  all $q$-uniformly convex bodies  is dense in $\mathfrak C(X)$ where $q=2$, if $1<p\leq2$, or $q=p$, \;if $p>2$.
  \end{corollary}
  Combining  Theorem \ref{12.9} with the characterization of $p$-uniformly convex spaces, we have the following.
  \begin{theorem}\label{12.10}
   Let  $X$ be a Banach space, and $p\in[2,+\infty)$. Then the following statements are equivalent.

   i) $X$ is $p$-uniformly convex;

   ii) $X$ has a $p$-uniformly convex body;

   iii) $\mathfrak C^{uc\text{,}p}(X)$ contains a dense  subset of $\mathfrak C(X)$.
  \end{theorem}
 \section{Uniformly smooth  convex bodies  of power type}
  Recall that for a Banach space $X$, its modulus of smoothness $\rho_X$ is defined by
  \[
   \rho_X(t)=\sup\{\|x+t y\|+\|x-t y\|-2:x,y\in S_X\},\;t\in[0,2).
   \]
  We say that $X$ is uniformly smooth of power type $q$ if there exists $C>0$ such that
   \[
    \rho_X(t)\leq Ct^q,\;\;t\in[0,2).
   \]
  It follows from Pisier's renorming theorem that for every super reflexive Banach space $X$, there exists $q\in(1,2]$ such that $X$ can be renormed to be uniformly smooth with a modulus of smooth of power type $q$.

  \begin{definition}\label{12.1}
   Let $X$ be a Banach space, $1<q\leq2$,  $B$ be a convex body of $X$ with  $0\in{\rm int}(B)$, and $p_B$ be the Minkowski functional generated by $B$.

   i) The modulus of smoothness of $B$ is defined by
   \begin{equation}\label{13.1}
    \rho_B(t)=\sup\Big\{p_B(x+t y)+p_B(x-ty)-2: \;x,y\in{\rm bd}(B)\Big\},\;0\leq t<\infty.
   \end{equation}

   ii) $B$ is said to be $q$-uniformly smooth, if there exists a constant $c>0$ such that
   \[\rho_B(t)\leq ct^q,\;\;\forall\;0<t<\infty;\]

   iii) A convex body $C\subset X$ is called  $q$-uniformly smooth, if there exists $x_0\in\text{int}(C)$ and a constant $c>0$ such that
   \[\rho_B(t)\leq ct^q,\;\;\forall\;0<t<\infty,\]
  where $B=C-x_0$.
  \end{definition}
  Analogous to Lemma \ref{11.202}, we have the following.
  \begin{lemma}\label{13.2}
  A convex body $B$ of a Banach space $X$ is $q$-uniformly smooth if and only if for every $x_0\in\text{int}(B)$ there exists $c>0$ such that
  \[\rho_{B_0}(t)\leq ct^q,\;\;\forall\;0<t<\infty,\]
  where $B_0=B-x_0$.
  \end{lemma}
  We use $\mathfrak{C}^{us\text{,}q}(X)$ (resp.,$\mathfrak C_0^{us\text{,}q}(X)$)  to denote the set of all $p$-uniformly smooth bodies (resp., all $p$-uniformly smooth bodies containing the origin), and by
   by $\mathfrak C^{us\text{,}q}_{00}(X)$ the set of all $q$-uniformly smooth bodies $B$ with $0\in{\rm int}(B)$.

 The following theorem follows from Lemma \ref{13.2}.
 \begin{theorem}\label{13.3}
 Let $X$ be a Banach space, and $1<q\leq2$. Then
  \[\mathfrak{C}^{us\text{,}q}(X)=X+\mathfrak C^{us\text{,}q}_{00}(X).\]
  \end{theorem}

 The following characterization of $q$-uniformly smooth convex bodies can be found in \cite[Theorem 5.7]{Ch2009}.
  \begin{theorem}\label{13.4}
   Let $B$ be a convex body  of a Banach space $X$ with $0\in{\rm int}(B)$. Then $B$ is $q$-uniformly smooth if and only if there exists $c>0$, such that for all $x,y\in X$,
   \begin{equation}\label{13.4.1}
    p^p_B(x+y)+p^p_B(x-y)-2p^p_B(x)\leq cp^p_B(y)
   \end{equation}
  \end{theorem}

   Let $C$ be a convex set of a Banach space $X$, and $f$ be a convex function defined on $C$.
   The function $f$ is said to be  {\it $q$- uniformly smooth }  if  there is $c>0$ such that for  all $x,y\in C$ with $x\pm y\in C$,
   \[
   f(x+y)+f(x-y)-2f(x)\leq c\|y\|^q.
   \]

   We should  mention that  the norm of a $q$-uniformly smooth Banach space is not a $q$-uniformly smooth function. Nevertheless, by Theorem \ref{13.4}, we have the following.

  \begin{theorem}\label{13.5}
   Let $X$ be a Banach space, and $q\in(1,2]$. Then a convex body $B\subset X$  with $0\in{\rm int}(B)$ is $q$-uniformly smooth   if and only if  $f=p_B^q$ is $q$-uniformly smooth. Consequently, $X$ is $q$-uniformly smooth if and only if $f=\|\cdot\|^q$ is $q$-uniformly smooth.
  \end{theorem}

  \begin{lemma}\label{13.6}
   Let $X$ be a Banach space, and $2\leq p<\infty$. If $g$ is a continuous $p$-uniformly convex function  on $X^*$, then $f\equiv\mathscr F^{-1}(g)$ is a continuous, $q$-uniformly smooth convex function  on $X$, where $q$ satisfies $\frac{1}{p}+\frac{1}{q}=1$.
  \end{lemma}
   \begin{proof}
   Since $g$ is a continuous $p$-uniformly convex function  on $X^*$, $X^*$ (hence, $X$) is super reflexive, and $f=\mathscr F^{-1}(g)$ continuous on $X$.
     Fix $x, h \in X$ and choose any $y_1 \in \partial f(x + h)$ and $y_2 \in \partial f(x - h)$. Note  $x^*\in\partial f(x)\;\Longleftrightarrow\; f(x)+g(x^*)=\langle x^*,x\rangle$. Then
    \[
     f(x + h) + g(y_1) = \langle x + h, y_1 \rangle,\;\text{and} ~
     f(x - h) + g(y_2) = \langle x - h, y_2 \rangle.
    \]
    Hence,
    \begin{equation}\label{13.5.1}
      f(x + h) + f(x - h) = \langle x, y_1 + y_2 \rangle + \langle h, y_1 - y_2 \rangle - g(y_1) - g(y_2).
    \end{equation}
    By the $p$-uniform convexity  of $g$, there is a constant $c$ (which is independent of $y_1, $ and $y_2$) such that
    \begin{equation}\label{13.5.2}
     g\left( \frac{y_1 + y_2}{2} \right) \le \frac{g(y_1) + g(y_2)}{2} - c\|y_1 - y_2\|^p.
    \end{equation}
    On the other hand, since $g=\mathscr F(f)$, we see that for any $y \in X^*$,
    \[
     f(x) \ge \langle x, y \rangle - g(y).
    \]
   Taking $y = \frac{y_1 + y_2}{2}$ in the inequality above, we obtain
    \begin{equation}\label{13.5.3}
      f(x) \ge \left\langle x, \frac{y_1 + y_2}{2} \right\rangle - g\left( \frac{y_1 + y_2}{2} \right).
    \end{equation}
    This and \ref{13.5.2} imply
    \[
     f(x) \ge \left\langle x, \frac{y_1 + y_2}{2} \right\rangle - \frac{g(y_1) + g(y_2)}{2} + c \|y_1 - y_2\|^p,
    \]
    Further,
    \begin{equation}\label{13.5.4}
    g(y_1) + g(y_2) \ge 2 \left\langle x, \frac{y_1 + y_2}{2} \right\rangle - 2f(x) + 2c \|y_1 - y_2\|^p.
    \end{equation}
    It follows from \ref{13.5.1} and \ref{13.5.4}
    \[
     \begin{aligned}
      f(x + h) + f(x - h) &\le \langle x, y_1 + y_2 \rangle + \langle h, y_1 - y_2 \rangle \\
                          &\quad - \left[ 2 \left\langle x, \frac{y_1 + y_2}{2} \right\rangle - 2f(x) + 2c \|y_1 - y_2\|^p \right] \\
                          &= 2f(x) + \langle h, y_1 - y_2 \rangle - 2c \|y_1 - y_2\|^p.
     \end{aligned}
    \]
    Let $z = y_1 - y_2$. Then
    \[
     f(x + h) + f(x - h) - 2f(x) \le \langle h, z \rangle - 2c \|z\|^p.
    \]
    Therefore,
    \[
     f(x + h) + f(x - h) - 2f(x) \le \sup_{z \in X^*} \left\{ \langle h, z \rangle - 2c \|z\|^p \right\}
    \]
    \[=\mathscr F(2c \|\cdot\|^p)=\frac{1}{q} (2c)^{-q/p} \|\cdot\|^q,\]
   where $1/p + 1/q = 1$. Let $K = \frac{1}{q} (2c)^{-q/p}$. Then
    \[
     f(x + h) + f(x - h) - 2f(x) \le K \|h\|^q,
    \]
   which shows that $f$ is $q$-uniformly smooth.
  \end{proof}
  Let $p\in[1,+\infty).$ Recall that for a Banch space $X$, we denote by $\mathscr H^p_{c}(X)$ the set of all positively homogeneous, continuous, convex, and coercive functions of degree $p$; by $\mathscr H^p_{uc,p}(X)$ the set of all continuous, positively homogeneous, $p$-uniformly convex, coercive functions of degree $p$; and by $\mathscr H^p_{us,p}(X)$ the set of all continuous, positively homogeneous, $p$-uniformly smooth, coercive convex functions of degree $p$.

  Now, we will use Lemma \ref{13.6} to prove the main  theorem of this section.

 \begin{theorem}\label{13.7}
  For every super reflexive Banach space $X$, there is $q\in(1,2]$ such that the set $\mathfrak C^{us\text{,}q}(X)$ of all $q$-uniformly smooth convex bodies   contains a dense  subset of $\mathfrak C(X)$.
 \end{theorem}

  \begin{proof}
   Since $X$ is super reflexive, it dual space $X^*$ is also super reflexive. By Theorem \ref{12.9}, there exists some $p\in[2,\infty)$ such that $\mathscr H^{p}_{uc,p}(X^*)$ contains a dense subset of $\mathscr H^{p}_{c}(X^*)$. Applying Lemma \ref{13.5} and the duality isomorphism (Lemma \ref{13.6}), we deduce that $\mathscr H^{q}_{us,q}(X)$ contains a dense  subset of $\mathscr H^{q}_{c}(X)$, where $\frac{1}{p}+\frac{1}{q}=1$. Finally,  Theorems \ref{5.4} and \ref{13.3} yield that $\mathfrak C^{us\text{-}q}(X)$ contains a dense subset of $\mathfrak C(X)$.
  \end{proof}
  As an immediate consequence of Theorem \ref{13.7}, we obtain the following result for Hilbert spaces.
 \begin{corollary}
  If $X$ is isomorphic to a subspace of an $L^p$-space for some $1<p<\infty$,  then $\mathfrak C^{us\text{,}q}(X)$ of all $q$-uniformly smooth convex bodies contains a dense  subset of $\mathfrak C(X)$, where $q=p$, if $1<p\leq2$, or $q=2$, \;if $p>2$.
 \end{corollary}

 Combining the aforementioned conclusions, we have the following theorem.
  \begin{theorem}\label{13.8}
   Let $X$ be a Banach space and $q\in(1,2]$. Then the following statements are equivalent.

   i) $X$ is $q$-uniformly smooth;

   ii) $X$ has a $p$-uniformly smooth convex body;

   iii) $\mathfrak C^{us,q}(X)$ contains a dense subset of $\mathfrak C(X)$.
  \end{theorem}
  \section{$p$-uniformly convex and $q$-uniformly smooth convex bodies}
  In this section, we will show that for every super reflexive Banach space $X$, there exist $p\in[2,+\infty)$ and $q\in(1,2]$ such that the set $\mathfrak C^{uc,p}(X)\bigcap\mathfrak C^{us,q}(X)$ of both $p$-uniformly convex and $q$-uniformly smooth convex bodies is dense in  $\mathfrak C(X)$, that is , the following theorem.

  \begin{theorem}
  Let $X$ be a super reflexive Banach space. Then there exist $p\in[2,+\infty)$ and $q\in(1,2]$ such that $\mathfrak C^{uc,p}(X)\bigcap\mathfrak C^{us,q}(X)$ is dense in $\mathfrak C(X)$, the cone of all convex
  bodies endowed with the Hausdorff metric.
 \end{theorem}
 \begin{proof}

   Since $X$ is  super reflexive, it follows from Enflo-Pisier Theorem that there is $p, p^\prime\in [2,\infty)$ such that $X$ can be renormed to be $p$-uniformly convex, and $X^*$ $p^\prime$-uniformly convex. Let $\|\cdot\|$ on $X$ and $\||\cdot\||^*$ be such equivalent norms. Let $\||\cdot\||$ be the dual norm of $\||\cdot\||^*$ on $X$.  Next, we will use  Asplund's everaging techneque \cite{Asp}, starting from $f_0=\frac{1}2\|\cdot\|^2$ and  $g_0=\frac{1}2\||\cdot\||^2$ to show that there is an equivalent $p$-uniformly convex norm $|\cdot|$ (i.e. $|\cdot|^p$ is a $p$-uniformly convex function) on $X$ so that its dual norm $|\cdot|^*$ is $p^\prime$-convex on $X^*$.  
   
   Let \[f_n=\frac{1}2(f_{n-1}+g_{n-1}), \;g_n=f_{n-1}\Box g_{n-1},\; n=1,2,\cdots,\] where $f\Box g$ is the inf-convolution of $f$ and $g$.
Then for all $n\in\N$, \[f_1\geq\cdots\geq f_{n-1}\geq f_n\geq g_n\geq g_{n-1}\geq\cdots\geq g_1,\]
and
\[\mathscr F(f_1)\leq\cdots\leq \mathscr F(f_{n-1})\leq\mathscr F(f_n)\leq \mathscr F(g_n)\leq \mathscr F(g_{n-1})\leq\cdots\leq \mathscr F(g_1).\]
 Note that $f_1^{\frac{p}2}=\Big(\frac{f_0+g_0}2\Big)^\frac{p}2$ is $p$-uniformly convex on $X$  and $(\mathscr F(g_1))^\frac{p^\prime}2=\Big(\frac{\mathscr F(f_0)+\mathscr F(g_0)}2\Big)^\frac{p^\prime}2$ is $p^\prime$-uniformly convex on $X^*$, and that $f_{n+1}=\frac{f_n+g_n}2$ and $\mathscr F(g_{n+1})=\frac{1}2(\mathscr F(f_{n})+\mathscr F(g_{n}))$. Then for each $n\in\mathbb N$, $(f_n)^\frac{p}2$ is $p$-uniformly convex on $X$ and $(\mathscr F(g_n))^\frac{p^\prime}2$ is $p^\prime$-uniformly convex on $X^*$. Monotonicity of both the sequences $\{f_n\}$ and $\{g_n\}$, and $\sup_{x\in B_X}|f_n(x)-g_n(x)|\rightarrow 0$ as $n\rightarrow\infty$ entail that $\lim_nf_n=\lim_ng_n\equiv h$ exist. By an argument of convexity we see that $\sqrt{h}$ is an equivalent $p$-uniformly convex norm on $X$ and $\sqrt{\mathscr F(h)}$ is an equivalent $p^\prime$-uniformly convex norm on $X^*$.
 
    By Theorems \ref{12.10}, the set $\mathfrak C^{uc,p}(X)$ of $p$-uniformly convex bodies of $X$  is dense in $\mathfrak C(X)$, and the set $\mathfrak C^{uc,p^\prime}(X^*)$ of $p^\prime$-uniformly convex bodies of $X^*$  is dense in $\mathfrak C(X^*)$. 
    
    Next, we will show for each convex body $B$ of $X$, and every $\eps>0$, there is a $p$-uniformly convex and $q$-uniformly smooth convex body $C$ so that $d_H(B,C)<\eps$, where $1<q\leq2$ satisfies $\frac{1}{p^\prime}+\frac{1}q=1$.
    
    Without loss of generality, we assume that the closed unit $B_X$ of $X$ is contained in $B$. For every $\eps>0$, by the fact we have just proven, there exist a $p$-uniformly convex body $D\subset B$ with $0\in\text{int}(C)$ and $d_H(B,D)<\eps/2$, and a $q$-uniformly smooth convex body $E\supset B$ with $d_H(B, E)<\eps/2$.
    Let $f_0=\frac{1}2p_D^2$ and $g_0=\frac{1}2p_E^2$. Then $f_0\geq\frac{1}2p_B^2\geq g_0$, and $d(f,g)=\sup_{x\in B_X}|f(x_0)-g_0(x)|<\eps.$  
    
    Now, we start with $f_0$ and $g_0$, and repeat the Asplund's averaging procedure above, i.e. let 
    \[f_n=\frac{1}2(f_{n-1}+g_{n-1}), \;g_n=f_{n-1}\Box g_{n-1},\; n=1,2,\cdots.\] 
    Then we obtain the limits of the both sequences $\{f_n\}$ and $\{g_n\}$ exist with   $\lim_nf_n=h=\lim_ng_n$, and satisfy that $h^\frac{p}2$ is $p$-uniformly convex, and $(\mathscr F(h))^\frac{p^\prime}2$ $p^\prime$-uniformly convex. Therefore, $h^\frac{q}2$ is $q$-uniformly smooth on $X$,  where $1<q\leq2$ satisfies $\frac{1}{p^\prime}+\frac{1}q=1$. Therefore, $C\equiv\{x\in X: h(x)\leq\frac{1}2\}$ is a $p$-uniformly convex and $q$-uniformly smooth convex body satisfying $d_H(B,C)<\eps$.
 \end{proof}

    Note that every $L_p$ ($1<p<\infty$) is  $2$-uniformly convex if $p\leq2$ and $q$-uniformly smooth ($\frac{1}p+\frac{1}q=1$) , and $p$-uniformly convex if $p>2$ and $q$-uniformly smooth ($\frac{1}p+\frac{1}q=1$). we have the following consequence.
    \begin{corollary}
    Suppose that a Banach space $X$ is isomorphic to a subspace of an $L_p$-space for some $1<p<\infty$, and that $q$ satisfies $\frac{1}p+\frac{1}q=1$.  
    
    i) if $p\leq2$, then every convex body can be approximated by both $2$-uniformly convex and $q$-uniformly smooth convex bodies; 
    
    ii) if $p>2$, then every convex body can be approximated by both $p$-uniformly convex and $q$-uniformly smooth convex bodies. 
    \end{corollary}
    
    \[\mathbf{Acknowlagements}\]
    
    The  named authors want to thank Professor Chunlan Jiang and Professor Liping Yuan for their very helpful suggestions on this paper.

\bibliographystyle{amsalpha}

\end{document}